\documentclass[11pt,a4paper]{amsart}

\usepackage{amssymb,amsmath,graphics,verbatim}
\usepackage{latexsym}
\usepackage{eucal}
\usepackage{a4wide}

\newtheorem{proposition}{Proposition}[section]
\newtheorem{theorem}{Theorem}[section]
\newtheorem{lemma}[theorem]{Lemma}
\newtheorem{prop}[theorem]{Proposition}
\newtheorem{coro}[theorem]{Corollary}
\newtheorem{remark}[theorem]{Remark}

\newtheorem{definition}[theorem]{Definition}
\newcommand{\mc}{\mathcal}
\newcommand{\cal}{\mathcal}
\newcommand{\rr}{\mathbb{R}}
\newcommand{\R}{\mathbb{R}}
\newcommand{\C}{\mathbb{C}}
\newcommand{\nn}{\mathbb{N}}
\newcommand{\cc}{\mathbb{C}}

\newcommand{\eps}{\epsilon}

\newcommand{\pl}{\partial}
\newcommand{\x}{\times}

\newcommand{\til}{\widetilde}
\newcommand{\bbar}{\overline}

\newcommand{\supp}{\textrm{supp}}
\newcommand{\cjd}{\rangle}
\newcommand{\cjg}{\langle}

\newcommand{\demi}{\frac{1}{2}}

\newcommand{\indic}{\operatorname{1\negthinspace l}}

\newcommand{\im}{\mathop{\hbox{\rm Im}}\nolimits}
\newcommand{\re}{\mathop{\hbox{\rm Re}}\nolimits}
\def\qed{\hfill$\square$\medskip}

\newcommand{\D}{\mathbb{D}}

\newcommand{\lt}[1]{{\mar{LT:#1}}}

\def\qed{\hfill$\square$\medskip}

\newcommand{\mar}[1]{{\marginpar{\sffamily{\scriptsize
        #1}}}}

\begin{document}
\title[The Reflection Principle and Calder\'on Problems with Partial Data]{The Reflection Principle and Calder\'on Problems with Partial Data}

\author{Leo Tzou}
\address{School of Mathematics and Statistics\\
Sydney University\\
}
\email{leo.tzou@gmail.com}

\maketitle
\begin{section}{Introduction}
Let $M_0$ be a smooth Riemann surface with boundary, equipped with a metric $g$.
A complex line bundle $E$ on $M_0$ has a trivialization $E\simeq M_0\x\cc$, thus there is a non-vanishing smooth section
$s:M_0\to E$,  and  a connection $\nabla$ on $E$ induces  
a complex valued $1$-form $iX$ on $M_0$ (where $i=\sqrt{-1}\in\cc$) defined by
$\nabla s=s\otimes iX$, which means that $\nabla (fs)=s\otimes (d+ iX)f$ if $d$ is the exterior derivative. 
The associated connection Laplacian ($*$ is the Hodge operator with respect to $g$) is the operator
\[\Delta^X := {\nabla^X}^*\nabla^X = -*(d* +i X\wedge *)(d +iX)\] 
acting on complex valued functions (sections of $E$). 
When $X$ is real valued, this operator is often called the \emph{magnetic Laplacian} associated 
to the magnetic field $dX$, and the connection $1$-form $X$ can be seen as to a connection $1$-form
on the principal bundle $M_0\x S^1$ by identifying $i\rr\subset \cc$ with the Lie algebra of $S^1$. This also corresponds 
to a Hermitian connection, in the sense that it preserves the natural Hermitian product on $E$.
Let $V$ be a complex valued  function on $M_0$ and assume that the $1$-form $X$ is real valued,
and consider the \emph{magnetic Schr\"odinger Laplacian}
associated to the couple $(X,V)$ 
\begin{equation}\label{defL}
L _{X,V}:=  {\nabla^X}^*\nabla^X + V= -*(d* + iX\wedge *)(d +iX)+V.
\end{equation}
If $H^s(M_0)$ denotes the  Sobolev space with $s$ derivatives in $L^2$ and $\Gamma\subset \partial M_0$ is an open subset such that $\partial M_0 \backslash \Gamma$ contains an open segment, we define the partial Cauchy data space of $L_{X,V}$ to be
{\small \begin{equation}\label{CL}
\mc{C}_{X,V, \partial M_0\backslash \Gamma} := \{ (u,\nabla^X_{\nu} u\mid_{\partial M_0\backslash\Gamma})\mid  u\in H^1(M_0), \supp(u\mid_{\partial M_0}) \subset \partial M_0\backslash \Gamma, L_{X,V}u=0\}
\end{equation}}where $\nu$ is the outward pointing unit normal vector field to $\pl M_0$ and 
$\nabla_\nu^Xu:=(\nabla^Xu)(\nu)$ . 
The first natural inverse problem is to see if the Cauchy data space determines the connection form $X$ and the potential $V$ uniquely,
and one easily sees that it is not the case since there are gauge invariances in the problem: for instance,
conjugating $L_{X,V}$ by $e^{f}$ with $f=0$ on $\pl M_0\backslash \Gamma$, one obtains the same partial Cauchy data space but with a Laplacian associated to the  connection $\nabla^{X+df}$, therefore it is not possible to identify $X$ but rather one should expect to recover the connection $\nabla^X$ modulo isomorphism.

It was shown in \cite{gafa} and \cite{albin guillarmou tzou} that, in the special case when $\Gamma = \emptyset$, the Cauchy data uniquely determines the connection $\nabla^X$ up to unitary bundle isomorphisms which are identity on the boundary and the potential $V$. This was done in \cite{gafa} through showing that the Cauchy data determines the integrals of $X$ along closed loops modulo integer multiples of $2\pi$. For planar domains, this result was first proved by Imanuvilov-Yamamoto-Uhlmann in \cite{IUY2} assuming only partial data measurement.

For these types of results in Euclidean domains of dimensions three and higher, we refer the readers to the works of Henkin-Novikov \cite{HeNo2}, Sun \cite{Sun1, Sun2}, Nakamura-Sun-Uhlmann in \cite{NaSuUh}, Kang-Uhlmann in \cite{KaUh}, and for partial data Dos Santos Ferreira-Kenig-Sj\"ostrand-Uhlmann in \cite{DKSU}. 
For simply connected planar domains, Imanuvilov-Yamamoto-Uhlmann in \cite{IUY2} deal with  
the case of general second order elliptic operators for partial data measurement, and Lai \cite{Lai} deals with the special case
of magnetic Schr\"odinger operator for full data measurement.

For $s\in \nn,p\in[1,\infty]$, let us denote by $W^{s,p}(M_0)$ and $W^{s,p}(M_0; T^*M_0)$ the Sobolev spaces consisting of functions and 1-forms respectively with $s$ derivatives in $L^p$. If $X_1, X_2 \in W^{3,p}(M_0; T^*M_0)$ and $V_1, V_2\in W^{2,p}(M_0)$ for $p$ large, we assume that the partial Cauchy data spaces for $L_{X_1,V_1}$ and $L_{X_2, V_2}$ agree
\begin{eqnarray}
\label{main assumption}
{\mathcal C}_{X_1, V_1, \partial M_0\backslash \Gamma}={\mathcal C}_{X_2, V_2, \partial M_0\backslash \Gamma}.
\end{eqnarray}
As the Cauchy data is invariant under the gauge transformation $X \mapsto X+ d\zeta$ for $\zeta \in W^{4,p}(M_0)\cap H^1_0(M_0)$, we may assume without loss of generality that 
\begin{eqnarray}
\label{boundary normal assumption}
 \iota_\nu(X_1 - X_2) = 0.
\end{eqnarray}
The main result of this paper is the following generalization of the results of \cite{gafa}:
\begin{theorem}
\label{main thm}
Let $X_1, X_2\in W^{3,p}(M_0; T^*M_0)$ be real-valued 1-forms and $V_1, V_2 \in W^{2,p}(M_0)$ be functions such that they satisfy \eqref{main assumption} and \eqref{boundary normal assumption}. Then there exists a non-vanishing function $\Theta$ with $\Theta\mid_{\partial M_0\backslash \Gamma} = 1$ such that $iX_1 = iX_2 + \Theta^{-1} d\Theta$ and $V_1 = V_2$.
\end{theorem}
To simplify the geometry it is sometimes convenient to consider larger $\Gamma$. As such we will prove the following auxiliary theorem.
\begin{theorem}
\label{main thm, larger Gamma}
Let $X_1, X_2\in W^{3,p}(M_0; T^*M_0)$ be real-valued 1-forms and $V_1, V_2 \in W^{2,p}(M_0)$ be functions such that they satisfy \eqref{main assumption} and \eqref{boundary normal assumption}. Then there exists a subset $\Gamma_0\subset \partial M_0$ containing $\Gamma$ with $\partial M_0\backslash \bar\Gamma_0$ a connected open segment $\partial M_0$, and a non-vanishing function $\Theta$ with $\Theta\mid_{\partial M_0\backslash \Gamma_0} = 1$ such that $iX_1 = iX_2 + \Theta^{-1} d\Theta$ and $V_1 = V_2$.
\end{theorem}
Note that, unlike Theorem \ref{main thm}, we may assume without loss of generality in Theorem \ref{main thm, larger Gamma} that $\partial M_0\backslash \Gamma$ consists of a small line segment along the boundary. The fact that Theorem \ref{main thm} follows from Theorem \ref{main thm, larger Gamma} is a simple exercise in unique continuation and gauge transformation.

An approach to treat this problem in the case when $X_1=X_2=0$ was developed in \cite{duke}. The technique was based on ideas of  \cite{IUY} and \cite{bu} of constructing CGO vanishing on $\Gamma$ whose phase is stationary at a prescribed point. One then applies stationary phase expansion at the critical points to extract point-wise information on the coefficients. 

There are two difficulties when applying this technique to prove Theorem \ref{main thm, larger Gamma}. First, the presence of first order terms in the boundary integral identity causes derivatives of the phase function to appear in the integrand and thus prevent one from obtaining the desired information at the critical points of the phase function. Second, one needs to construction CGO with higher regularity via a "shifted" Carleman estimate. The standard methods of shifting loses track of the boundary structure (see e.g. \cite{DKSU}) and therefore it is not clear how one can construct CGO with $H^1_{scl}$ estimates and at the same time vanish on $\Gamma$. Chung in \cite{chung} resolved the "shifting" issue in $\R^n$ for $n\geq 3$ and our approach is partially inspired by his ideas. In the planar case, Imanuvilov-Uhlmann-Yamamoto in \cite{IUY2} overcame these difficulties by direct computation and our method, based more on geometry, differs significantly from their approach.

The first difficulty is resolved through the use of a new boundary integral identity:
\begin{prop}
\label{new boundary integral id}

Under the assumptions of \eqref{main assumption} and \eqref{boundary normal assumption}, if one sets $A_j := \pi_{0,1} X_j$, then there exists an open boundary component $\Gamma_0$ containing $\bar\Gamma$ with $\partial M_0 \backslash\bar\Gamma_0$ an open segment of $\partial M_0$, such that one can find non-vanishing functions $F_{A_j} \in W^{2,p}(M_0)\cap W^{4,p}_{loc}(M_0)$ solving
\begin{eqnarray}\label{boundary conditions for conjugation factor}F_{A_j}^{-1}\bar\pl F_{A_j} =i A_j,\ \ \ |F_{A_j}|\mid_{\Gamma_0} = 1\ \ j=1,2\ \ {\rm with}\ \ F_{A_1}\mid_{\partial M_0\backslash \Gamma_0} = F_{A_2}\mid_{\partial M_0\backslash \Gamma_0}.\end{eqnarray}
Furthermore, for any pair of $\{F_{A_1}, F_{A_2}\}$ satisfying \eqref{boundary conditions for conjugation factor} and solutions $u_j$ to
\[L_{X_j, V_j} u_j = 0\ \ \ \ u_j\mid_{\Gamma_0} =0\]
 one has
\begin{eqnarray}
\label{new boundary integral equation}
0=\int_{M_0} \langle (|F_{A_1}|^{-2} - |F_{A_2}|^{-2})\bar\pl \tilde u_1, \bar\pl \tilde u_2\rangle +  \frac{1}{2}\langle (Q_2 |F_{A_2}|^{2} -  Q_1 |F_{A_1}|^{2})\tilde u_1,\tilde u_2\rangle \end{eqnarray}
where $\tilde u_j = F_{A_j}u_j$ and $Q_j = *dX_j + V_j$.
\end{prop}
Note that as both solutions are differentiated only by $\bar\pl$ we can then construct CGO (in Section 5) which are compatible with this differential operator so that the difficulty of the phase function appearing in the integrand would not occur. Arriving at \eqref{new boundary integral equation} requires one to see how assumption \eqref{main assumption} leads to the existence of a holomorphic extension of the function ${F_{A_1}}{F_{A_2}}^{-1}\mid_{\partial M_0\backslash\Gamma}$ for any non-vanishing solutions of $F_{A_j}^{-1}\bar\pl F_{A_j} =i A_j$. This is achieved by considering the double of Riemann surfaces and exploit the symmetry of the holomorphic extension problem under reflection. 

The second difficulty, the one of "shifting" the Carleman estimate, will be treated again by using the reflection principle. In this case we double the bordered Riemann surface and extend the harmonic Carleman weight with reflection principle. On the doubled surface we "shift" the Carleman estimate with the semiclassical pseudodifferential operator $\langle hD\rangle^{-1}$ as in \cite{DKSU}. We then use symmetry to see that this shift operation on the doubled surface actually leaves a large portion of the original boundary intact.

In addition to highlighting the geometric nature of this problem, the approach outlines here allows one to extending the setting of \cite{IUY2} to general surfaces. Furthermore, the program described here can be applied to study a wide range of inverse problems involving the connection Laplacian. In a series of forthcoming articles we will use the approach outlined here to treat:
\begin{enumerate}
\item The partial Cauchy data problem for the Hodge Laplacian on surfaces (see \cite{chung salo tzou} for the higher dimensional case), \item The partial Cauchy data problem for Dirac systems (the full data case was considered in \cite{albin guillarmou tzou}), \item Inverse scattering on surfaces in the presence of magnetic potentials (the special case when $X_1 = X_2 =0$ was considered in \cite{GST}). 
\end{enumerate}
The systematic approach developed here will facilitate future discussions which naturally follow the identifiability result we prove - that of stability, analytic reconstruction, and numerical reconstruction.  
\end{section}
\begin{section}{Harmonic and Holomorphic Morse Functions on a Riemann Surface}

\subsection{Riemann surfaces}
We start by recalling few elementary definitions and results about Riemann surfaces, see for instance \cite{FK} for more details.
Let $(M,g)$ be a compact connected smooth Riemannian surface with boundary $\pl M$. The surface $M$ can be considered as 
a subset of a compact Riemannian surface, for instance by taking the double of $M$.

The conformal class of $g$ on the closed surface $M$ induces a structure of closed Riemann surface, i.e. a closed surface equipped with a complex structure via holomorphic charts $z_\alpha:U_{\alpha}\to \cc$. The Hodge star operator $\star$ acts on the cotangent bundle $T^*M$, its eigenvalues are 
$\pm i$ and the respective eigenspace $T_{1,0}^*M:=\ker (\star+i{\rm Id})$ and $T_{0,1}^*M:=\ker(\star -i{\rm Id})$
are sub-bundle of the complexified cotangent bundle $\cc T^*M$ and the splitting $\cc T^*M=T^*_{1,0}M\oplus T_{0,1}^*M$ holds as complex vector spaces.
Since $\star$ is conformally invariant on $1$-forms on $M$, the complex structure depends only on the conformal class of $g$.
In holomorphic coordinates $z=x+iy$ in a chart $U_\alpha$,
one has $\star(udx+vdy)=-vdx+udy$ and
\[T_{1,0}^*M|_{U_\alpha}\simeq \cc dz ,\quad T_{0,1}^*M|_{U_\alpha}\simeq \cc d\bar{z}  \]
where $dz=dx+idy$ and $d\bar{z}=dx-idy$. We define the natural projections induced by the splitting of $\cc T^*M$ 
\[\pi_{1,0}:\cc T^*M\to T_{1,0}^*M ,\quad \pi_{0,1}: \cc T^*M\to T_{0,1}^*M.\]
The exterior derivative $d$ defines the De Rham complex $0\to \Lambda^0\to\Lambda^1\to \Lambda^2\to 0$ where $\Lambda^k:=\Lambda^kT^*M$
denotes the real bundle of $k$-forms on $M$. Let us denote $\cc\Lambda^k$ the complexification of $\Lambda^k$, then
the $\pl$ and $\bar{\pl}$ operators can be defined as differential operators 
$\pl: \cc\Lambda^0\to T^*_{1,0}M$ and $\bar{\pl}:\cc\Lambda_0\to T_{0,1}^*M$ by 
\begin{equation}\label{defddbar}
\pl f:= \pi_{1,0}df ,\quad \bar{\pl}:=\pi_{0,1}df,
\end{equation}
they satisfy $d=\pl+\bar{\pl}$ and are expressed in holomorphic coordinates by
\[\pl f=\pl_zf\, dz ,\quad \bar{\pl}f=\pl_{\bar{z}}f \, d\bar{z}.\]  
with $\pl_z:=\demi(\pl_x-i\pl_y)$ and $\pl_{\bar{z}}:=\demi(\pl_x+i\pl_y)$.
Similarly, one can define the $\pl$ and $\bar{\pl}$ operators from $\cc \Lambda^1$ to $\cc \Lambda^2$ by setting 
\[\pl (\omega_{1,0}+\omega_{0,1}):= d\omega_{0,1}, \quad \bar{\pl}(\omega_{1,0}+\omega_{0,1}):=d\omega_{1,0}\]
if $\omega_{0,1}\in T_{0,1}^*M$ and $\omega_{1,0}\in T_{1,0}^*M$.
In coordinates this is simply
\[\pl(udz+vd\bar{z})=\pl v\wedge d\bar{z},\quad \bar{\pl}(udz+vd\bar{z})=\bar{\pl}u\wedge d{z}.\]
There is a natural operator, the Laplacian acting on functions and defined by 
\[\Delta f:= -2i\star \bar{\pl}\pl f =d^*d \]
where $d^*$ is the adjoint of $d$ through the metric $g$ and $\star$ is the Hodge star operator mapping 
$\Lambda^2$ to $\Lambda^0$ and induced by $g$ as well.

\subsection{Maslov Index and Boundary value problem for the $\overline\partial$ Operator}
In this subsection we consider the setting where $M$ is an oriented Riemann surface with boundary $\partial M$ and $M_0'$ is a submanifold of $M$ such that $\partial M\cap \partial M_0' \neq \emptyset$. Denote by $\Gamma_0' \subset \partial M$ an open subset of $\partial M$ which compactly contains $\partial M\cap \partial M_0'$.  We assume in addition that $\partial M \backslash \bar\Gamma_0'$ contains an open set.
 
Following \cite{mcduff} (see also \cite{duke}), we adopt the following notations:
let $E\to M$ be a complex line bundle with complex structure $J : E\to E$ and let $D : C^{\infty}(M,E) \to C^\infty(M,T^*_{0,1}\otimes E)$ be a Cauchy-Riemann operator with smooth coefficients on $M$, acting on sections of the bundle $E$. Observe that in the case when $E=M\x \cc$ is the trivial line bundle with the natural complex structure on $M$, then $D$ can be taken to be the operator $\overline\partial$ introduced in \eqref{defddbar}. 
For $q>1$, we define
\[D_F : W^{\ell,q}_F(M,E) \to W^{\ell-1,q}(M, T_{0,1}^*M\otimes E)\]
where $F\subset E\mid_{\pl M}$ is a totally real subbundle (i.e. a subbundle such that $JF \cap F$ is the zero section) 
and $D_F$ is the restriction of $D$ to the $L^q$-based Sobolev space with $\ell$ derivatives and boundary condition $F$
\[W^{\ell,q}_F(M,E) := \{\xi\in W^{\ell,q}(M,E) \mid \xi(\partial M) \subset F\}.\]
The boundary Maslov index for a totally real subbundle $F\subset E_{\pl M}$ of a complex vector bundle 
is defined in generality in Appendix C.3 of \cite{mcduff}, we only recall the definition in our setting 
\begin{definition}\label{maslovindexdef}
Let $E=M\x \cc$ and $\pl M=\sqcup_{j=1}^m \pl_i M$ be a disjoint union of $m$
circles. The boundary Maslov index $\mu(E,F)$ is the degree of the map $\rho \circ \Lambda:\pl M\to \pl M$ where 
\[\Lambda|_{\pl_i M}: S^1\simeq \pl_i M\to{\rm GL}(1,\cc)/{\rm GL}(1,\rr)\] 
is the natural map assigning to $z\in S^1$ the totally real subspace $F_z\subset \cc$,
where ${\rm GL}(1,\cc)/{\rm GL}(1,\rr)$ is the space of totally real subbundles of $\cc$, and 
$\rho: {\rm GL}(1,\cc)/{\rm GL}(1,\rr) \to S^1$ is defined by $\rho(A.{\rm GL}(1,\rr)):=A^2/|A|^2$.
\end{definition}

In this setting, we have the following boundary value Riemann-Roch theorem stated in \cite{mcduff}:
\begin{theorem}
\label{boundaryrr}
Let $E \to M$ be a complex line bundle over an oriented compact Riemann surface with boundary and $F\subset E\mid_{\pl M}$ be a totally real subbundle.  Let $D$ be a smooth Cauchy-Riemann operator on $E$ acting on $W^{\ell,q}(M,E)$ for some $q>1$ 	and $\ell\in\nn$. 
Then\\
1) The following operators are Fredholm 
\[D_F : W^{\ell,q}_F(M, E) \to W^{\ell-1,q}(M,T_{0,1}^*M\otimes E)\]
\[D_F^* : W^{\ell,q}_F(M, T_{0,1}^*M\otimes E) \to W^{\ell-1,q}(M, E).\]
2) The real Fredholm index of $D_F$ is given by
\[{\rm Ind}(D_F) = \chi(M) + \mu(E,F)\]
where $\chi(M)$ is the Euler characteristic of $M$ and $\mu(E,F)$ is the boundary Maslov index of the subbundle $F$.\\
3) If $\mu(E,F) <0$, then  $D_F$ is injective, while if  $\mu(E,F) + 2\chi(M) >0$ the operator $D_F$ is surjective.
\end{theorem}

As an application, we obtain the following (here and in what follows, $H^m(M):=W^{m,2}(M)$):
\begin{prop}\label{surject}
(i) For $q>1$ and $k\in \nn_0$, there exists a bounded operator \[\bar\partial^{-1} : W^{k,q}(M,T^{*}_{0,1}M) \to \{u\in W^{k+1,q}(M) \mid u\mid_{\Gamma'_0} \in \R\}\] satisfying $\bar\partial \bar\partial^{-1} = Id$.\\
(ii) If $\chi \in C_0^\infty(M)$ is supported in a complex charts $U$ 
bi-holomorphic to a bounded open set $\Omega\subset \cc$ with complex coordinate $z$,  then as operators
\[\bar{\pl}^{-1}\chi ={\chi}' \bar{T} \chi+K\]
where $\chi' \in C_0^\infty(U)$ are such that $\chi'\chi=\chi$,  $K$ has a smooth 
kernel on $M\x M$ and $\bar{T}$ is given in the complex coordinate $z\in U\simeq \Omega$ by 
\[\bar{T}(fd\bar{z})=\frac{1}{\pi}\int_{\cc}\frac{f(z')}{z-z'}dz_1'dz_2'\]
where $dv_g(z)=\alpha^2(z)dz_1dz_2$ is the volume form of  $g$ in the chart.
\\
(iii) For $m> 1/2$, let $f\in H^{m}(\pl M)$ be a real valued function, then there exists a holomorphic function $v\in H^{m+\demi}(M)$ such that
${\rm Re}(v)|_{\Gamma_0'}=f$. Furthermore, $v$ can be chosen so that $\|v\|_{H^{m + \demi}(M)} \leq C_{m} \|f\|_{H^m(M)}$.\\
(iv) For $k\in\nn$ and $q>1$, the space of $W^{k,q}(M)$ holomorphic functions on $M$ which are real valued on $\Gamma_0'$ is infinite dimensional.
\end{prop}
\noindent{\bf Proof}. (i) Let $L\in\nn$ be arbitrary large and let us identify the boundary as a disjoint union of circles $\partial M= \coprod_{i = 1}^m \pl_iM$ where each $\pl_iM \simeq S^1$. Since $\Gamma_0'$ can be chosen so that $\partial M \backslash\Gamma_0'$ is as small as we like, it is sufficient to assume that $\partial M \backslash\Gamma_0'$  is a connected non-empty open segment of $\pl_1M = S^1$, and which can thus be defined in a coordinate $\theta$ (respecting the orientation of the boundary) by 
$\partial M \backslash\Gamma_0'= \{\theta \in S^1 \mid 0< \theta< 2\pi/k\}$ for some integer $k$.
Define the totally real subbundle of $F\subset E|_{\partial M} = \coprod_{j = 1}^m (\pl_jM \times \C)$ by the following: on 
$\pl_1M \simeq S^1$ parametrized by $\theta \in [0,2\pi]$, define $F_\theta = e^{ia(\theta)}\R\subset \cc$, 
where $a: [0,2\pi] \to \R$ is a smooth nondecreasing function such that $a(\theta) = 0$ in a neighbourhood $[0,\eps]$ of $0$, $a(2\pi/k) = 2L\pi$ 
for some $L\in\nn$, and $a(\theta) = 2L\pi$ for all $\theta > 2\pi/k$. In particular $F_z=\rr$ is constant for $z\notin \partial M \backslash\Gamma_0' $. 
For the rest of $\pl_2M,.., \pl_m M$, we  just let $F|_{\pl_iM} = S^1 \times \R$. 
The map $\Lambda$ in Definition \ref{maslovindexdef} is then given on $\pl_1M$ by $\Lambda(e^{i\theta})=
e^{ia(\theta)}{\rm GL}(1,\rr)$ and on $\pl_2M,\dots,\pl_mM$ by $\Lambda(e^{i\theta})=
e^{i\theta}{\rm GL}(1,\rr)$, therefore the Maslov index $\mu(E,F)$ is given by the degree of the map 
$e^{i\theta}\to e^{2ia(\theta)}$  on $S^1$, and this is given by $(a(2\pi)-a(0))/2\pi=2L$.
By theorem \ref{boundaryrr}, $D_F$ is surjective if $2\chi(M)+2L>0$. Since $L$ can be taken 
as large as we want this establishes the solvability assertion of (i).

To obtain the estimate, we fix $L$ large enough so that $2\chi(M) + 2L >0$ and consider the splitting given by $W^{k+1,q}(M) = \ker D_F + (\ker D_F)^{\perp}$. By taking a projection one sees that for all $\omega\in W^{k,q}(M)$ there exists a unique element $u\in (\ker D_F)^\perp$ such that $\bar\partial u = \omega$. Therefore we conclude that $D_F: (\ker D_F)^\perp \to W^{k,q}(M,T^{*}_{0,1}M)$ is a linear bijection and the uniform boundedness principle gives the desired estimate.\\

(ii) Observe that $\bar{\pl}^{-1}\bar{\pl}-1$ maps $W^{k,p}_F(M)$ into $\ker \bar{\pl}\cap W^{k,p}_F(M)$ which is a finite dimensional space
spanned by some smooth functions $\psi_1,\dots \psi_n$ (by elliptic regularity) on $M$. Assuming that $(\psi_j)_j$ is an orthonormal basis in $L^2$, this implies that, on $W^{1,2}_F(M)$
\[\bar{\pl}^{-1}\bar{\pl}=1-\Pi \, \textrm{ where }\, \Pi= \sum_{k=1}^n\psi_k\cjg\cdot,\psi_k\cjd_{L^2(M)}.\]
Now we also have
\[\bar{\pl}\chi'\bar{T}\chi= \chi+[\bar{\pl},\chi']\bar{T}\chi\]
and the last operator on the right has a smooth kernel in view of $\chi\nabla\chi'=0$ 
and the fact that $T$ has a smooth kernel outside the diagonal $z=z'$. 
Now since $\chi'\in C_0^\infty(M)\subset W^{1,2}_F(M)$, we can multiply by $\bar{\pl}^{-1}$ on the left of the last identity 
and obtain 
\[\bar{\pl}^{-1}\chi=\chi'\bar{T}\chi-\Pi\chi'\bar{T}\chi- \bar\pl^{-1}[\bar{\pl},\chi']\bar{T}\chi.\]  
The last two operator on the right have a smooth kernel on $M\x M$, in view 
of the smoothness of $\psi_k$ and the kernel of $[\bar{\pl},\chi']\bar{T}\chi$, 
and since $\bar\pl^{-1}$ maps $C_0^{\infty}(M,T_{0,1}^*M)$ to $C^{\infty}(M)$.\\

(iii) Let $w\in H^{m+\demi}(M)$ be a real function with boundary value $f$ on $\pl M$, then by (i) there exists $R\in H^{m+1/2}(M)$ with $\|R\|_{H^{m+\demi}(M)} \leq C \|w\|_{H^{m+\demi}(M)} \leq C \|f\|_{H^m(\partial M)}$ such that $i\bar{\pl}R=-\bar{\pl}w$ and $R$ purely real on $\Gamma_0'$, thus $v:=iR+w$ is holomorphic such that ${\rm Re}(v)=f$ on $\Gamma_0'$.\\ 

(iv) Taking the subbundle $F$ as in the proof of (i), we have that $\dim \ker D_F=\chi(M)+2L$ if $L$ satisfies $2\chi(M)+2L>0$, 
and since $L$ can be taken as large as we like, this concludes the proof. 
\qed

\begin{lemma}
\label{control the zero}
Let $\{p_0, p_1,..,p_n\}\subset M$ be a set of $n+1$ disjoint points. Let $c_1,\dots, c_K\in\cc$, $N\in\nn$, 
and let $z$ be a complex coordinate near $p_0$ such that $p_0=\{z=0\}$. 
Then if $p_0\in{\rm int}(M)$, there exists a holomorphic function $f$ on $M$ with zeros of order at least $N$ at each $p_j$,
such that $f$ is real on $\Gamma_0'$ and $f(z)=c_0 + c_1z +...+ c_K z^K+ O(|z|^{K+1})$ in the coordinate $z$. If 
$p_0\in \pl M$, the same is true except that $f$ is not necessarily real on $\Gamma_0'$.
\end{lemma}
\noindent{\bf Proof}.
First, using linear combinations and induction on $K$, it suffices to prove the Lemma for any $K$ and $c_0=\dots=c_{K-1}=0$, which we now show.  
Consider the subbundle $F$ as in the proof of (i) in Proposition \ref{surject}. 
The Maslov index $\mu(E,F)$ is  given by $2L$ and so for each $N\in\nn$, one can take $L$ large enough to have $\mu(F,E) + 2\chi(M) \geq 2N(1+n)$. Therefore by Theorem \ref{boundaryrr} the dimension of the kernel of $\overline\partial_F$ will be greater than $2(n+1)N$. 
Now, since for each $p_j$ and complex coordinate $z_j$ near $p_j$, 
the map $u\to (u(p_j),\pl_{z_j}u(p_j),\dots,\pl_{z_j}^{N-1}u(p_j))\in \cc^{N}$ is linear, 
this implies that there exists a non-zero element $u\in \ker D_F$ which has zeros of order at least $N$ 
at all $p_j$.

First, assume that $p_0\in {\rm int}(M)$ and we want the desired Taylor expansion at $p_0$ in the coordinate $z$.  
In the coordinate $z$, one has $u(z) = \alpha z^M+O(|z|^{M+1})$ for some $\alpha\not=0$ and $M\geq N$. 
Define the function $r_K(z) = \chi(z)\frac{c_K}{\alpha} z^{-M+K}$ 
where $\chi(z)$ is a smooth cut-off function supported near $p_0$ and
which is $1$ near $p_0=\{z=0\}$. Since $M \geq N> 1$, this function has a pole at $p_0$ and trivially extends smoothly to 
$M\backslash \{p_0\}$, which we still call $r_K$. Observe that the function is holomorphic in a neighbourhood of $p_0$ but not at $p_0$ where it is only meromorphic, so that in $M\setminus \{p_0\}$, $\overline\partial r_K$ is a smooth and compactly supported section of $T^*_{0,1}M$ 
and therefore trivially extends smoothly to $M$ (by setting its value to be $0$ at $p_0$) to a one form denoted $\omega_K$. 
By the surjectivity assertion in Corollary \ref{surject}, 
there exists a smooth function $R_K$ satisfying $\overline\partial R_K = - \omega_K$ and that $R_K|_{\Gamma_0'} \in \R$. 
We now have that $R_K + r_K$ is a holomorphic function on $M\backslash\{p_0\}$ 
meromorphic with a pole of order $M-K$ at $p_0$, and in coordinate $z$ one has $z^{M-K}(R_K(z)+r_K(z))= c_K+O(|z|)$. 
Setting $f_K= u(R_K + r_K)$, we have the desired holomorphic function. Note that $f$ also vanish to order $N$ at all 
$p_1,\dots,p_n$ since $u$ does. This achieves the proof.

Now, if $p_0\in \pl M$ we can consider a slightly larger manifold $M'$ containing $M$ and we apply the the result above. \qed\\
We conclude this subsection with the following estimate for the operator $\bar\partial^{-1}e^{2i\psi/h}$.    
\begin{lemma}\label{estimate1}
Let $U$ be an open subset compactly contained in $M$ and for $q,p\in [1,\infty]$. Let $\psi$ be a real valued smooth Morse function on $M$ and let $\bar{\pl}^{-1}_\psi:=\bar{\pl}^{-1}e^{2i\psi/h}$ where $\bar{\pl}^{-1}$ 
is the right inverse of $\bar{\pl}:W^{1,p}(M)\to L^{p}(T^*_{0,1}M)$ constructed in Proposition \ref{surject}. 
Let $q\in(1,\infty)$ and $p>2$,  then there exists $C>0$ independent of $h$ such that for all 
$\omega\in W^{1,p}_0(U,T^*_{0,1}M)$  
\begin{equation}\label{q<2} 
||\bar{\pl}^{-1}_\psi \omega||_{L^q(M)}\leq Ch^{2/3}||\omega||_{W^{1,p}(M,T^*_{0,1}M)} \, \, \,{\rm if }\, 1 \leq q<2 
\end{equation}
\begin{equation}\label{q>=2}
||\bar{\pl}^{-1}_\psi \omega||_{L^q(M)}\leq Ch^{1/q}||\omega||_{W^{1,p}(M,T^*_{0,1}M)} \,\,\, {\rm if }\, 2\leq q\leq p.
\end{equation}
There exists $\eps>0$ and $C>0$ such that for all $\omega\in W_c^{1,p}(M,T^*_{0,1}M)$
\begin{equation}\label{q=2}
||\bar{\pl}^{-1}_\psi \omega||_{L^2(M)}\leq Ch^{\demi+\eps}||\omega||_{W^{1,p}(M,T^*_{0,1}M)}. 
\end{equation}

\end{lemma}
\begin{proof}
Observe that the estimate \eqref{q=2} is a direct corollary of \eqref{q>=2} and \eqref{q<2} by using interpolation.
We recall the Sobolev embedding $W^{1,p}(M)\subset C^\alpha(M)$ for 
$\alpha\leq 1-2/p$ if $p>2$, and we shall denote by $\bar T$ the Cauchy-Riemann inverse of $\pl_{\bar{z}}$ in $\cc$:
\[\bar T(fd\bar z):=\frac{1}{\pi}\int_{\cc}\frac{f(\xi)}{z-\xi}d\xi_1d\xi_2\]
where $\xi=\xi_1+i\xi_2$. If $\Omega, \Omega'\subset \cc$ are bounded open sets, then 
the operator $\indic_{\Omega'}T$ maps $L^p(\Omega)$ to $L^p(\Omega')$. 
Since $\omega$ is  compactly supported in a chart $U$ biholomorphic to a bounded domain $\Omega\subset \cc$, 
 and since the estimates will be localized, we can assume with no loss of generality that $\psi$ has only one critical point, say $z_0\in \Omega$ (in the chart).  
The expression of $\bar{\pl}^{-1}_\psi (fd\bar{z})$ in complex local coordinates in the chart $\Omega$ satisfies 
\[ \bar{\pl}^{-1}_\psi (f(z)d\bar{z})= \chi(z)T(e^{-2i\psi/h}f) + K(e^{-2i\psi/h}fd\bar{z})\]
where $K$ is an operator with smooth kernel and $\chi\in C_0^\infty(\cc)$ is identically $1$ on $U$.\\

Let us first prove \eqref{q<2}. Let $\chi_\delta\in C_0^\infty(\cc)$ be a function which is equal to $1$ for $|z-z_0|>2\delta$ and 
to $0$ in $|z-z_0|\leq \delta$, where  $\delta>0$ is a parameter that will be chosen later (it will depend on $h$).
Using Minkowski inequality, one can write when $q<2$
\begin{equation}\label{1-varphi}
\begin{split}
||\chi T((1-\chi_\delta)e^{-2i\psi/h}f)||_{L^q(\cc)} 
\leq & \int_{\Omega}\Big|\Big|\frac{\chi(\cdot)}{|\cdot -\xi|}\Big|\Big|_{L^q(\cc)}|(1-\chi_\delta(\xi))f(\xi)| d\xi_1d\xi_2\\
\leq & C||f||_{L^\infty}\int_{\Omega}|(1-\chi_\delta(\xi))| d\xi_1d\xi_2
\leq C\delta^2||f||_{L^\infty}.
\end{split}\end{equation}
On the support of $\chi_\delta$, we observe that since $\chi_\delta=0$ near $z_0$, we can use 
\[T(e^{-2i\psi/h}\chi_\delta f)=\demi ih[e^{-2i\psi/h}\frac{\chi_\delta f}{\bar{\pl}\psi}-T(e^{-2i\psi/h}\bar{\pl}(\frac{\chi_\delta f}{\bar{\pl}\psi}))]\] 
and the boundedness of $T$ on $L^q$ to deduce that for any $q<2$
\begin{equation}\label{varphi}
\begin{split}
||\chi T(\chi_\delta e^{-2i\psi/h}f)||_{L^q(\cc)} 
\leq & Ch \Big(||\frac{\chi_\delta f}{\bar{\pl}\psi}||_{L^q} +
||\frac{f\bar{\pl}\chi_\delta}{\bar{\pl}\psi}||_{L^q}+||\frac{\chi_\delta \bar{\pl}f}{\bar{\pl}\psi}||_{L^q}+
||\frac{f\chi_\delta}{(\bar{\pl}\psi)^2}||_{L^q}
\Big).
\end{split}\end{equation}
The first term is clearly bounded by $\delta^{-1}\|f\|_{L^\infty}$ due to the fact that $\psi$ is Morse. For the last term, observe that since $\psi$ is Morse, $\frac{1}{|\partial \psi|} \leq \frac{c}{|z-z_0|}$ near $z_0$, therefore
\[||\frac{f\chi_\delta}{(\bar{\pl}\psi)^2}||_{L^q} \leq C\|f\|_{L^\infty} (\int_\delta^1 r^{1-2q}dr)^{1/q} \leq C \delta^{\frac{2}{q} -2}\|f\|_{L^\infty}.\]
The second term can be bounded by $||\frac{f\bar{\pl}\chi_\delta}{\bar{\pl}\psi}||_{L^q}\leq\|f\|_{L^\infty} ||\frac{\bar{\pl}\chi_\delta}{\bar{\pl}\psi}||_{L^q}$. Observe that while $\|\frac{\bar{\pl}\chi_\delta}{\bar{\pl}\psi}\|_{L^\infty}$ grows like $\delta^{-2}$, $\bar{\pl}\chi_\delta$ is only supported in a neighbourhood of radius $2\delta$. Therefore we obtain
\[||\frac{f\bar{\pl}\chi_\delta}{\bar{\pl}\psi}||_{L^q} \leq \delta^{2/q -2} \|f\|_{L^\infty}.\]
The third term can be estimated by
\[||\frac{\chi_\delta \bar{\pl}f}{\bar{\pl}\psi}||_{L^q}\leq C||\bar{\pl}f||_{L^p} ||\frac{\chi_\delta}{\bar{\pl}\psi}||_{L^\infty}\leq 
C\delta^{-1}||\bar{\pl}f||_{L^p}. \] 
Combining these four estimates with \eqref{varphi} we obtain
\[||\chi T(\chi_\delta e^{-2i\psi/h}f)||_{L^q(\cc)}\leq h\|f\|_{W^{1,p}}(\delta^{-1} + \delta^{2/q -2}).\] 
Combining this and \eqref{1-varphi} and optimizing by taking $\delta=h^{1/3}$, we deduce that 
\begin{equation}
\label{singular part q<2}
||\chi T( e^{-2i\psi/h}f)||_{L^q(\cc)}\leq h^{2/3}\|f\|_{W^{1,p}}
\end{equation}
if $q<2$.
We now move on to the smoothing part given by $K(e^{-2i\psi/h}f)$. Take $\chi$ to be a compactly supported function in $\Omega$ such that it is equal to $1$ on the support of $f$, we see that $K(e^{2i\psi/h}f) = K(e^{-2i\psi/h}(f - \chi f(z_0)) + f(z_0)K(e^{-2i\psi/h}\chi)$. By applying stationary phase, we easily see that $\|f(z_0)K(e^{-2i\psi/h}\chi)\|_{L^q} \leq Ch \|f\|_{C^0}$ for any $q\in [1,\infty]$. 
For the first term, we write $\til{f}:=f-\chi f(z_0)$ and we integrate by parts to get, for some smoothing operator $K'$ 
\[K(e^{-2i\psi/h}\til{f}) = hK'(e^{-2i\psi/h}\til{f}) +\frac{h}{2i} K\Big(e^{-2i\psi/h}\pl_z\Big(\frac{\til{f}}{\pl_z\psi}\Big)\Big).\]
By the fact that $K$ and $K'$ are smoothing, we see that for all $k\in \nn$
\[\|K(e^{2i\psi/h}\til{f})\|_{C^k} \leq h C\Big(\|f\|_{L^\infty} + \Big\|\pl_z\Big(\frac{\til{f}}{\pl_z\psi}\Big)\Big\|_{L^1}\Big)\]
Using the fact that $\psi$ is Morse, the Sobolev embedding $W^{1,p}\subset C^{\alpha}$ for $\alpha=1-2/p$ 
and  $\til{f}(z_0)=0$, we can estimate the last term by $C \|f\|_{W^{1,p}}$ if $p>2$. Therefore,
\begin{eqnarray}
\label{smoothing estimate}
\|K(e^{2i\psi/h}f)\|_{L^q} \leq C h\|f\|_{W^{1,p}}
\end{eqnarray}
for any $q\in [1,\infty]$ and $p>2$. Combining \eqref{smoothing estimate} and \eqref{singular part q<2} we see that \eqref{q<2} is established.\\

Let us now turn our attention to the case when $\infty>q\geq 2$, one can use the boundedness of $T$ on $L^q$ and thus 
\begin{equation}\label{casq>2}
||\chi T((1-\chi_\delta)e^{-2i\psi/h}f)||_{L^q(\cc)}\leq  ||(1-\chi_\delta)e^{-2i\psi/h}f||_{L^q(\Omega)}\leq C\delta^{\frac{2}{q}}||f||_{L^\infty}.
\end{equation}
Now since $\chi_\delta=0$ near $z_0$, we can use 
\[T(e^{-2i\psi/h}\chi_\delta f)=
\demi ih[e^{-2i\psi/h}\frac{\chi_\delta f}{\pl_{\bar{z}}\psi}-T(e^{-2i\psi/h}\pl_{\bar{z}}(\frac{\chi_\delta f}{\pl_{\bar{z}}\psi}))]\] 
and the boundedness of $T$ on $L^q$ to deduce that for any $q\leq p$, \eqref{varphi} holds again with all the terms satisfying the same estimates as before so that 
\[\|T(e^{-2i\psi/h}\chi_\delta f)\|_{L^q} \leq Ch\|f\|_{W^{1,p}}(\delta^{2/q -2} + \delta^{-1}) \leq Ch\delta^{2/q-2}\|f\|_{W^{1,p}} \]
since now $q \geq 2$. Now combine the above estimate with \eqref{casq>2} and take $\delta = h^\demi$ we get
\[\|T(e^{-2i\psi/h} f)\|_{L^q} \leq h^{1/q} \|f\|_{W^{1,p}}\] 
for $2 \leq q\leq p$. The smoothing operator $K$ is controlled by \eqref{smoothing estimate} for all $q\in [1,\infty]$ and therefore we obtain \eqref{q>=2}.
\end{proof}

\subsection{Morse holomorphic functions with prescribed critical points} 
The main result of this section is the following 
\begin{proposition}
\label{criticalpoints}
Let $\hat p$ be an interior point of $M$ and $\epsilon >0$ small. 
Then there exists a holomorphic function $\Phi$ on $M$ which is Morse on $M$ (up to the boundary) and real valued on $\Gamma_0'$, 
which has a critical point $p'$ at distance less than $\eps$ from $\hat p$ and such that 
${\rm Im}(\Phi(p'))\not=0$. 
\end{proposition}
Let $\mc{O}$ be a connected open set of $M^D$ such that $\bar{\mc{O}}$ is a smooth surface with boundary, with $\bar{M}\subset \bar{\mc{O}}\subset M^D$ and $\Gamma_0'\subset\pl\bar{\mc{O}}$.
Fix $k>2$ a large integer, we denote by $C^k(\bar{\mc{O}})$ the Banach space of $C^k$ real valued functions on $\bar{\mc{O}}$. 
Then the set of harmonic functions on $\bar{\mc{O}}$ which are in the Banach space
$C^{k}(\bar{\mc{O}})$ (and smooth in $\mc{O}$ by elliptic regularity) 
is the kernel of the continuous map $\Delta:C^k(\bar{\mc{O}})\to C^{k-2}(\bar{\mc{O}})$, 
and so it is a Banach subspace of $C^k(\bar{\mc{O}})$. 
The set $\mc{H}\subset C^k(\bar{\mc{O}})$ of harmonic functions $u$ in $C^k(\bar{\mc{O}})$ such there exists 
$v\in C^{k}(\bar{\mc{O}})$ harmonic with $u+iv$ holomorphic on $\mc{O}$ is a Banach subspace of $C^{k}(\bar{\mc{O}})$ of 
finite codimension. Indeed, let $\{\gamma_1,..,\gamma_N\}$ be a homology basis for $\mc{O}$, then
\[\mc{H}=\ker L , \textrm{ with } L: \ker\Delta\cap C^k(\bar{\mc{O}})\to \cc^N \textrm{ defined by }
L(u):=\Big(\frac{1}{\pi i}\int_{\gamma_j}\pl u\Big)_{j=1,\dots,N}.\]
For all $\tilde \Gamma_0\subset \pl \cal O$ such that the complement of $\tilde\Gamma_0$ contains an open subset, 
we define 
\[\mc{H}_{\tilde\Gamma_0} := \{u \in\mc{H} ; u|_{\tilde\Gamma_0} =0\}.\]
We now show
\begin{lemma}\label{morsedense}
The set of functions $u\in\mc{H}_{\tilde\Gamma_0}$ which are Morse in $\mc{O}$ 
is residual (i.e. a countable intersection of open dense sets) 
in $\mc{H}_{\tilde\Gamma_0}$ with respect to the $C^k(\bar{\mc{O}})$ topology.
\end{lemma}
\noindent{\bf Proof}. We use an argument very similar to those used by Uhlenbeck \cite{Uh}.
We start by defining $m: \mc{O}\times \mc{H}_{\tilde\Gamma_0}\to T^*\mc{O}$ by $(p,u) \mapsto (p,du(p))\in T_p^*\mc{O}$. 
This is clearly a smooth map, linear in the second variable, moreover $m_u:=m(.,u)=(\cdot, du(\cdot))$ is  
Fredholm since $\mc{O}$ is finite dimensional. The map $u$ is a Morse function if and only if 
$m_u$ is transverse to the zero section, denoted $T_0^*\mc{O}$, of $T^*\mc{O}$, ie. if 
\[\textrm{Image}(D_{p}m_u)+T_{m_u(p)}(T_0^*\mc{O})=T_{m_u(p)}(T^*\mc{O}),\quad \forall p\in \mc{O} \textrm{ such that }m_u(p)=(p,0),\]
which is equivalent to the fact that the Hessian of $u$ at critical points is 
non-degenerate (see for instance Lemma 2.8 of \cite{Uh}). 
We recall the following transversality theorem (\cite[Th.2]{Uh}):
\begin{theorem}\label{transversality}
Let $m : X\times \mc{H}_{\tilde\Gamma_0} \to W$ be a $C^k$ map, where $X$, $\mc{H}_{\tilde\Gamma_0}$, and $W$ are separable Banach manifolds 
with $W$ and $X$ of finite dimension. Let $W'\subset W$ be a submanifold such that $k>\max(1,\dim X-\dim W+\dim W')$.
If $m$ is transverse to $W'$  then the set 
$\{u\in \mc{H}_{\tilde\Gamma_0}; m_u \textrm{ is transverse to } W'\}$ is dense in $\mc{H}_{\tilde\Gamma_0}$, more precisely 
it is a residual set. 
\end{theorem}
We want to apply it with $X:=\mc{O}$, $W:=T^*\mc{O}$ and $W':=T^*_0\mc{O}$, and the map $m$ is defined above. 
We have thus proved Lemma \ref{morsedense} if one can show that $m$ is transverse to $W'$. 
Let $(p,u)$ such that $m(p,u)=(p,0)\in W'$. Then identifying $T_{(p,0)}(T^*\mc{O})$ with $T_p\mc{O}\oplus T^*_p\mc{O}$, one has
\[D_{(p,u)}m(z,v)=(z,dv(p)+{\rm Hess}_p(u)z)\]
where ${\rm Hess}_pu$ is the Hessian of $u$ at the point $p$, viewed as a linear map from $T_p\mc{O}$ to $T^*_p\mc{O}$. 
To prove that $m$ is transverse to $W'$
we need to show that $(z,v)\to (z, dv(p)+{\rm Hess}_p(u)z)$ is onto from $T_p\mc{O}\oplus \mc{H}_{\tilde\Gamma_0}$ to $T_p\mc{O}\oplus T^*_p\mc{O}$, which 
is realized for instance if the map $v\to dv(p)$ from $\mc{H}_{\tilde\Gamma_0}$ to  $T_p^*\mc{O}$ is onto.
But from Lemma \ref{control the zero}, we know that there exist holomorphic functions $v$ and $\tilde v$ on $\mc{O}$  such that $v$ and $\tilde v$ are purely real on $\tilde\Gamma_0$. Clearly the imaginary parts of $v$ and $\tilde v$ belong to $\mc{H}_{\tilde\Gamma_0}$. Furthermore, for a given complex coordinate $z$ near $p=\{z=0\}$, we can arrange them to have series expansion $v(z)=z+ O(|z|^2)$ and ${\tilde v}(z)=iz + O(|z|^2)$ around the point $p$. We see, by coordinate computation of the exterior derivative of ${\rm Im}(v)$ and ${\rm Im}(\til{v})$, that 
 $d\, {\rm Im}(v)(p)$ and $d \,{\rm Im}(\tilde v)(p)$ are linearly independent at the point $p$. This shows 
 our claim and ends the proof of Lemma \ref{morsedense} by using Theorem \ref{transversality}.\qed\\
We now proceed to show that the set of all functions $u\in \mc{H}_{\tilde\Gamma_0}$ such that $u$ has no degenerate critical points on $\tilde\Gamma_0$ is also residual.

\begin{lemma}
\label{nonvanishing normal derivative}
For all $p\in \tilde\Gamma_0$ and $k\in\nn$, there exists a holomorphic function $u\in C^{k}(\bar{\mc{O}})$, such that ${\rm Im}(u)|_{\tilde\Gamma_0}=0$  and 
$\pl u(p) \neq 0$.
\end{lemma}
\noindent{\bf Proof}. The proof is quite similar to that of Lemma \ref{control the zero}. By Lemma \ref{control the zero}, we can choose a holomorphic function $v\in C^k(\bar{\mc{O}})$ such that $v(p)=0$ and ${\rm Im}(v)|_{\tilde\Gamma_0}=0$, then either $\pl v(p)\not=0$ and we are done, or $\pl v(p)=0$. 
Assume now the second case and let $M\in\nn$ be the order of $p$ as a zero of $v$. By Riemann mapping theorem, there is a conformal
mapping from a neighbourhood $U_p$ of $p$ in $\bar{\mc{O}}$ to a neighbourhood $\{|z|<\eps, {\rm Im}(z)\geq 0\}$ of the 
real line ${\rm Im}(z)=0$ in $\cc$, and one can assume that $p=\{z=0\}$ in these complex coordinates. 
Take $r(z)=\chi(z)z^{-M+1}$ where $\chi\in C_0^\infty(|z|\leq \eps)$ is a real valued function with $\chi(z)=1$ in $\{|z|<\eps/2\}$. Then 
$\bar{\pl} r$ 
vanishes in the pointed disc $0<|z|<\eps/2$ and it is a compactly supported smooth section of $T^*_{1,0}\bar{\mc{O}}$ outside, it can thus be 
extended trivially to a smooth section of $T^*_{1,0}\bar{\mc{O}}$ denoted by $\omega$. 
We can then use (i) of Corollary \ref{surject}:
there is a function $R$ such that $\bar{\pl} R=-\omega$ and ${\rm Im}(R)|_{\tilde\Gamma_0}=0$, and
so $\bar{\pl}(R+r)=0$ in $\mc{O}\setminus \{p\}$ and $R+r$ is real valued on $\tilde\Gamma_0$ (remark that $r$ is real valued on $\tilde\Gamma_0$)
and has a pole at $p$ of order exactly $M-1$. We conclude that $u:=v(R+r)$ satisfies the desired properties, it 
vanishes at $p$ but with non zero complex derivative at $p$. 
\qed

\begin{lemma}
\label{degeneracy condition}
Let $\tilde\Gamma_0\subset \partial \mc{O}$ be an open set of the boundary. Let $\phi : \mc{O} \to \R$ be a harmonic function with $\phi|_{\tilde\Gamma_0} = 0$. Let $p\in\tilde\Gamma_0$ be a critical point of $\phi$, then it is nondegenerate 
if and only if $\partial_\tau\partial_\nu u \neq 0$ where $\pl_\tau$ and $\pl_\nu$ denote respectively the tangential and normal 
derivatives along the boundary.
\end{lemma}
\noindent{\bf Proof}. By Riemman mapping theorem, there is a conformal transformation mapping a neighbourhood of $p$ in $\bar{\mc{O}}$
to a half-disc $D:=\{|z|<\eps,{\rm Im}(z)\geq 0\}$ and $\pl\bar{\mc{O}}=\{{\rm Im}(z)=0\}$ near $p$. Denoting $z=x+iy$, one has
$(\pl_x^2+\pl_y^2)\phi=0$ in $D$ and $\pl^2_x\phi|_{y=0}=0$, which implies $\pl_y^2\phi(p)=0$. Since $\pl_\nu=e^{f}\pl_y$ and 
$\pl_\tau=e^{f}\pl_x$ for some smooth function $f$, and since $d\phi(p)=0$, the conclusion is then straightforward.
\qed\\
Let $N^{*}\pl \bar{\mc{O}}$ be the conormal-bundle of $\pl\bar{\mc{O}}$ and $N^*\tilde\Gamma_0$ be the restriction of this bundle to $\tilde\Gamma_0$. Denote the zero sections of these bundles respectively by $N^*_0\pl\bar{ \mc{O}}$ and $N^{*}_0\tilde\Gamma_0$. 
We now define the map
\[b : \tilde\Gamma_0 \times \mc{H}_{\tilde\Gamma_0}  \to N^{*}\tilde\Gamma_0, \quad b(p,u):=(p,\pl_\nu u).\] 
For a fixed $u\in \mc{H}_{\tilde\Gamma_0}$, we also define  $b_u(\cdot) := b(\cdot, u)$. Simple computations yield the
\begin{lemma}
\label{mixed partial non-vanishing}
Suppose that $p\in\tilde\Gamma_0$ is such that $\partial_\nu u(p) = 0$, then $\partial_\tau\partial_\nu u(p)\neq 0$ if and only if
\[{\rm Image}(D_pb_u) + T_{(p,0)}(N_0^{*}\tilde\Gamma_0) = T_{(p,0)}(N^{*}\tilde\Gamma_0).\]
\end{lemma}
\noindent{\bf Proof}. This can be seen by the fact that for all $p\in \tilde\Gamma_0$ such that $b_u(p) = (p,0)$, 
\[D_pb_u : T_p\tilde\Gamma_0\to T_{(p,0)}(N^{*}\tilde\Gamma_0) \simeq T_p\tilde\Gamma_0 \oplus N^{*}_p\tilde\Gamma_0\] 
is given by $w\mapsto (w, \partial_\tau\partial_\nu u(p)w)$.\qed\\

At a point $(p,u)$ such that $b(p,u) = 0$, a simple computation yields that the differential $D_{(p,u)}b : T_p\tilde\Gamma_0 \times \mc{H}_{\tilde\Gamma_0} \to T_{(p,\pl_\nu u(p))}(N^{*}\tilde\Gamma_0)$ is given by $(w,u') \mapsto (w, \partial_\tau\partial_\nu u(p)w + \partial_\nu u'(p))$. This observation combined with Lemma \ref{nonvanishing normal derivative} shows that for all $(p,u) \in \tilde\Gamma_0\times \mc{H}_{\tilde\Gamma_0}$ such that $b(p,u) = (p,0)$, $b$ is transverse to $N^{*}_0\tilde\Gamma_0$ at $(p,0)$. Now we can apply Theorem \ref{transversality} with $X = \tilde\Gamma_0$, $W = N^{*}\tilde\Gamma_0$ and $W' = N^{*}_0\tilde\Gamma_0$ we see that the set $\{u\in \mc{H}_{\tilde\Gamma_0}; b_u \textrm{ is transverse to }N^{*}_0\tilde\Gamma_0\}$ is residual in $\mc{H}_{\tilde\Gamma_0}$. In view of Lemmas \ref{degeneracy condition}, we deduce the 
\begin{lemma}
\label{boundarymorse}
The set of functions $u\in \mc{H}_{\tilde\Gamma_0}$ such that $u$ has no degenerate critical point on $\tilde\Gamma_0$ is residual in $\mc{H}_{\tilde\Gamma_0}$.
\end{lemma}
Observing the general fact that finite intersection of residual sets remains residual, the combination of Lemma \ref{boundarymorse} and Lemma \ref{morsedense} yields
\begin{coro}
\label{boundarymorse dense}
The set of functions $u\in\mc{H}_{\tilde\Gamma_0}$ which are Morse in $\mc{O}$ and have no degenerate critical points on $\tilde\Gamma_0$ 
is residual in $\mc{H}_{\tilde\Gamma_0}$ with respect to the $C^k(\bar{\mc{O}})$ topology. In particular, it is dense.
\end{coro}
We are now in a position to give a proof of the main proposition of this section.\\\\
\noindent{\bf Proof of Proposition \ref{criticalpoints}}. As explained above, choose $\mc{O}$ in such a way that $\bar{\mc{O}}$ is a smooth surface with boundary, containing $M$, that $\Gamma_0'\subset\partial \mc{O}$ and $\mc{O}$ contains $\partial M \backslash \overline{\Gamma_0'}$. Let $\tilde\Gamma_0$ be an open subset of the boundary of $\bar{\mc{O}}$ such that the closure of $\Gamma_0'$ is contained in $\tilde\Gamma_0$ and $\partial\bar{\mc{O}}\backslash\overline{\tilde\Gamma_0}\not=\emptyset$. Let $\hat p$ be an interior point of $M$. By lemma \ref{control the zero}, there exists a holomorphic function $f = u + i v$ on $\bar{\mc{O}}$ such that $f$ is purely real on $\tilde\Gamma_0$, $v(\hat p)= 1$, and $df(\hat p) = 0$ (thus $v\in \mc{H}_{\tilde\Gamma_0}$).

By Corollary \ref{boundarymorse dense}, there exist a sequence $(v_j)_j$ of Morse functions $v_j \in \mc{H}_{\tilde\Gamma_0} $ 
such that $v_j\to v$ in $C^k(M)$ for any fixed $k$ large. By Cauchy integral formula, there exist harmonic conjugates $u_j$ of $v_j$ such that $f_j := u_j + i v_j \to f$ in $C^k(M)$. Let $\eps>0$ be small and let $U\subset\mc{O}$ 
be a neighbourhood containing $p$ and no other critical points of $f$,
and with boundary a smooth circle of radius $\eps$. In complex local coordinates near $\hat p$, 
we can identify $\pl f$ and $\pl f_j$ to holomorphic functions on an open set of $\cc$.
Then by Rouche's theorem, it is clear that $\pl{f_j}$ has precisely one zero in $U$ and $v_j$ never vanishes in $U$ if $j$ is large enough.

Fix $\Phi$ to be one of the $f_j$ for $j$ large enough. By construction, $\Phi$ is Morse in $\mc{O}$ and has no degenerate critical points on $\overline{\Gamma_0'} \subset\tilde\Gamma_0$. 
We notice that, since the imaginary part of $\Phi$ vanishes on all of $\tilde\Gamma_0$, it is clear from
the reflection principle applied after using the Riemann mapping theorem (as in the proof of Lemma \ref{degeneracy condition})
that no point on $\overline{\Gamma_0'}\subset \tilde\Gamma_0$ can be an accumulation point for critical points. Now $\partial M \backslash \overline{\Gamma_0'}$ is contained in the interior of $\mc{O}$ and therefore no points on  $\partial M \backslash \overline{\Gamma_0'}$ can be an accumulation point of critical points. Since $\Phi$ is Morse in the interior of $\mc{O}$, there are no degenerate critical points on $\partial M \backslash \overline{\Gamma_0'}$. This ends the proof.\qed

\begin{subsection}{Doubling of Riemann Surfaces}
\label{sec:double}
We describe the construction of a double of a bordered Riemann surface outlined in \cite{FK}. Let $M$ and $M'$ be two copies of a bordered Riemann surface. We construct the closed surface $M^D := \overline{M \cup M'}$ by identifying points $p\in \partial M$ with its copy $p'\in \partial M'$. We take in the interior of $M$ the existing holomorphic coordinates while on $M'$ the holomorphic coordinates are precisely the complex conjugate of those on $M$. To construct coordinate charts along the boundary $\partial M$, if $U$ is a small neighbourhood in $M^D$ containing $p\in \partial M$ such that $U \cap \partial M$ is an open segment we take a holomorphic chart which maps $U\cap M$ conformally to the upper half plane such that $U\cap \partial M$ is mapped to a segment of the real axis. We can then apply the reflection principle to obtain a holomorphic coordinate chart around $p\in \partial M$. 

Let $M$ be a bordered Riemann surface which is isometric to the flat cylinder $([0,\epsilon] \times S^1, dt^2 + d\theta^2)$ near each of its boundary components. If $q_0\in \partial M$, define $M'_0\subset M$ by removing a small interior closed half-disk around $q_0$ of radius $\delta >0$ and let $\Gamma'$ be defined by ${\Gamma'} := \partial M_0' \cap \partial M$. If one denote by $\dot M^D :=M^D\backslash \bar B_\delta(q_0)$ with $B_\delta(q_0) := \{p\in M^D\mid d(p,q_0) <\delta\}$, then one has that $M_0' = \dot M^D\cap M$. That is, $M_0'$ is half of the surface obtained by removing a whole disk from $M^D$.

On every doubled Riemann surface $M^D$ there exists an anti-conformal involution $R$ satisfying $R(M) = M'$ and is the identity on the boundary $\partial M$. Since the metric $g$ on $M$ is assumed to be of the form $dt^2 + d\theta^2$ near $\partial M$, it extends smoothly to a metric on $M^D$ by the relation $R^*g = g$. It is easily checked that if $\Phi$ is a holomorphic function on $M_0'$ satisfying the boundary condition $\Phi\mid_{\Gamma'} \in \R$, then $\Phi$ extends to be a holomorphic function on $\dot M^D$ by the relation ${(R^* \Phi)} = \bar\Phi$. Similarly, if $\eta$ is a holomorphic 1-form with boundary condition $\iota^*_{\partial M_0'}\eta \mid_{\Gamma'} \in \R$, then $\eta$ extends to be a holomorphically to $\dot M^D$ by the relation $R^*\eta = \bar \eta$.

Conversely, if $\Phi$ is a holomorphic function on $\dot M^D$, we say it is conjugate even/odd if ${(R^* \Phi)} = \pm\bar\Phi$ and we adopt the same terminology for holomorphic forms. It is easily seen that the set of even holomorphic functions/1-forms are precisely the reflected ones described above.

\end{subsection}

\begin{subsection}{Boundary Values of Meromorphic Functions}
In this section we characterize the boundary value of holomorphic/meromorphic functions on the surface $\dot M^D$. These characterizations will be useful in boundary identification and in proving Proposition \ref{new boundary integral id}. We begin by stating a well-understood orthogonality condition for boundary values of holomorphic functions (see eg. \cite{gafa}). 
\begin{proposition}
\label{boundary value of holomorphic}
Let $f\in W^{2-\frac{1}{p},p}(\partial \dot M^D)$ be a complex valued function. Then $f$ is the restriction of a holomorphic function which is differentiable up to the boundary if and only if
\[\int_{\partial \dot M^D} f i_{\partial \dot M^D}^*\eta = 0\]
for all 1-forms $\eta\in C^{\infty}(\dot M^D ; T^*_{1,0}\dot M^D)$ satisfying $\bar\partial\eta = 0$.
\end{proposition}
We would like to generalize this statement to that of meromorphic functions with prescribed poles of certain order. As such we consider the following
\begin{lemma}
\label{meromorphic condition 1}
Let $\{p_0,.. p_N\} \subset \dot M^D\cup \partial \dot M^D$ be a discrete set of points. If $f\in W^{2-\frac{1}{p},p}(\partial \dot M^D)$ is a complex valued function satisfying
\[\int_{\partial \dot M^D} f i_{\partial \dot M^D}^*\eta = 0\]
for all holomorphic 1-forms $\eta\in C^{\infty}(\dot M^D ; T^*_{1,0}\dot M^D)$ with the property $\eta(p_j) = 0$ to $k$-th order, 
then $f$ is the restriction of a meromorphic function which is smooth up to the boundary and whose only poles lie in the interior points $\{p_0,..,p_N\} \cap \dot M^D$. Furthermore the poles are of order at most $k$.
\end{lemma}
\noindent{\bf Proof.} Let $a$ be a holomorphic function which is smooth up to the boundary with isolated zeros on $\dot M^D\cup \partial \dot M^D$ such that $a$ vanishes to exactly $k$-th order at $\{p_0,.., p_N\}$. Such functions can be constructed by compactly embedding $\dot M^D$ into a slightly larger surface with boundaries and apply Lemma \ref{control the zero}. If $f\in W^{2-\frac{1}{p},p}(\partial \dot M^D)$ is a complex function satisfying the hypothesis then one has 
\[\int_{\partial \dot M^D} (af) i_{\partial \dot M^D}^*\eta = 0\]
for all holomorphic 1-forms $\eta$. By Proposition \ref{boundary value of holomorphic} we have that $af \in W^{2-\frac{1}{p},p}(\partial \dot M^D)$ extends to a holomorphic function which we denote by $G_a $. 
Clearly, \[f = \frac{G_a}{a}\mid_{\partial \dot M^D} \in W^{2-\frac{1}{p},p}(\partial \dot M^D)\] is the restriction of the meromorphic function $\frac{G_a}{a}$ and since the zeros of $a$ are isolated, this meromorphic function is continuous up to the boundary. As such, the singularities of $\frac{G_a}{a}$ are precisely the interior zeros of $a$.

Let us now consider another holomorphic function $a'$ with isolated zeroes vanishing exactly to $k$-th order at $\{p_0,.., p_N\}$. By using Lemma \ref{control the zero}, we may construct $a'$ in such a way that $a$ and $a'$ do not have common zeroes in the interior other than $\{p_0,.., p_N\}\cap \dot M^D$. We repeat the above argument for $a'$ to show that $f =\frac{G_{a'}}{a'}\mid_{\partial \dot M^D}  \in W^{2-\frac{1}{p},p}(\partial \dot M^D)$ for some holomorphic function $G_{a'}$.

Unique continuation for meromorphic functions forces the identity $\frac{G_{a'}}{a'} = \frac{G_a}{a}$. The fact that the only common interior zeroes for $a$ and $a'$ are $\{p_0,.., p_N\}\cap M_0'$ ensures that they are the only poles and that they are of order at most $k$. Thus we conclude that $f$ extends to a meromorphic function differentiable up to the boundary whose only poles are $\{p_0,.., p_N\} \cap \dot M^D$ of degree at most $k$.\qed

Observe that if $R$ is the involution defined in Section \ref{sec:double}, then every holomorphic function $\Phi$ and 1-form $\eta$ can be decomposed into their conjugate even and odd part by writing \[\Phi = \frac{\Phi +\overline{(R^* \Phi)}}{2} + \frac{\Phi - \overline{(R^* \Phi)}}{2}\ \  {\rm and}\ \  \eta = \frac{\eta +\overline{(R^* \eta)}}{2} + \frac{\eta - \overline{(R^* \eta)}}{2}.\] As one can transform between conjugate even and odd functions via multiplication with $i\in \C$, one has that a smooth function $f\in W^{2-\frac{1}{p},p}(\partial \dot M^D)$ satisfies $\int_{\partial \dot M^D} f \iota_{\partial\dot M^D}\eta = 0$ for all conjugate {\em even} holomorphic 1-forms vanishing to $k$-th order at $\{p_1,..,p_N, R(p_1),.., R(p_N)\} \subset \dot M^D\cup \partial \dot M^D$ iff 
$\int_{\partial \dot M^D} f \iota_{\partial\dot M^D}\eta = 0$ for all holomorphic 1-forms vanishing to $k$-th order at the same points.

This discussion combined with Lemma \ref{meromorphic condition 1} gives the following condition for being the boundary value of a meromorphic function on $\dot M^D$
\begin{lemma}
\label{perpendicular to all conjugate even}
Let $f\in W^{2-\frac{1}{p},p}(\partial \dot M^D)$ and $\{p_1,..,p_N, R(p_1),..,R(p_N)\}$ be a discrete set of points in $\dot M^D\cup \partial \dot M^D$. The function $f$ is the boundary value of a meromorphic function in $\dot M^D$ with poles at $\{p_1,..,p_N, R(p_1),..,R(p_N)\}\cap \dot M^D$ of at most order $k$ if $\int_{\partial \dot M^D} f \iota^*_{\partial\dot M^D}\eta = 0$ for all conjugate even holomorphic 1-forms $\eta$ vanishing to order $k$ at $\{p_1,..,p_N, R(p_1),..,R(p_N)\}$.
\end{lemma}

\end{subsection}
\end{section}


\begin{section}{Shifted Carleman Estimates and $H^1$ Solvability}
In this section, we prove a shifted Carleman estimate on a Riemann surface using harmonic Morse weights. The estimate will have boundary conditions similar to the ones established in \cite{chung}.  We show the following estimate for $M_0'$, $M$, and $\Gamma'$ described in Section \ref{sec:double}:
\begin{proposition}
\label{shifted carleman estimate with boundary}
Let $\varphi: M\to\R$ be a $C^k(M)$ harmonic Morse function for $k$ large such that $\partial_\nu \varphi \mid_{\Gamma'_0} = 0$ for some open subset $\Gamma_0'\subset\partial M$ compactly containing $\Gamma'$.  For all $X\in W^{1,\infty}(M)$, $V\in L^\infty(M)$ there exists $h_0>0$ such that for all $u \in C^\infty_0(M_0')$ and $h \leq h_0$ we have
\[\|e^{-\varphi/h}h^2 L_{X,V} e^{\varphi/h} u\|_{H^{-1}_{scl}(M)} \geq Ch\big( \sqrt{h} \|u\| + \| d\varphi u\| \big)\] 
\end{proposition}
Note that since $\Delta_{e^Kg} = e^{K} \Delta_g$ it suffices to prove Proposition \ref{shifted carleman estimate with boundary} for a conformal representative of $g$ which is isometric to the flat cylinder near $\partial M$. The important feature in Proposition \ref{shifted carleman estimate with boundary} is that $\Gamma'$ is the common boundary component of $M_0'$ and $M$. This allows us to deduce the following semiclassical solvability while controlling the solution on a part of the boundary.
\begin{coro}
\label{H1 solvability}
Let $\varphi$ be as in Proposition \ref{shifted carleman estimate with boundary}. Then for all $f\in L^2(M_0')$ there exists a solution $u\in H_0^1(M)$ of the boundary value problem 
\[e^{-\varphi/h} L_{X,V} e^{\varphi/h} u = f\ \text{in}\ M_0',\ u\mid_{\Gamma'} = 0,\]
satisfying the estimate $\|u\| + \|h du\| \leq \sqrt{h} \|f\|$.
\end{coro}


We start the proof by modifying the weight as follows:
Let $\Gamma_0'\subset \partial M$ be an open subset compactly containing $\Gamma'$ so that $\partial M \backslash \Gamma_0'$ contains on open subset. If $\varphi_0:=\varphi: M \to \R$ is a real valued harmonic Morse function with critical points $\{p_1,\dots,p_N\}$ in $M \cup \partial M$ and $\partial_\nu\varphi_0 \mid_{\Gamma_0'} =0$, we 
let $\varphi_j:M \to \R$ be harmonic functions with boundary condition $\partial_\nu\varphi_j \mid_{\Gamma_0'} =0$ such that $p_j$ is not a critical point of $\varphi_j$ for $j = 1,\dots,N$. Their existence is ensured by  Lemma \ref{control the zero}.
For all $\epsilon >0$, we define the {\em convexified} weight $\varphi_{\eps} := \varphi - \frac{h}{2\epsilon}(\sum_{j = 0}^N|\varphi_j|^2)$. By Lemma \ref{nonvanishing normal derivative} we can choose $\varphi_j$ such that $\partial_\nu\varphi_j = 0$ on $\Gamma_0'$. 

As the normal derivatives of $\varphi_j$ along $\Gamma_0'$ all vanish, the even extensions of $\varphi_j$ to the double $M^D$ (which we denote again by $\varphi_j$) are harmonic on some connected bordered surface $M^D_\delta\subset M^D$ which compactly contains $\dot M^D$. We note that if $\varphi_0$ is Morse on $M\cup \partial M$, then its extension is Morse on $M^D_\delta$.
\begin{subsection}{Shifted Estimate on $M^D$} In this section let $M$, $M^D$, and the metric $g$ be as described in the construction given in Section \ref{sec:double}. We prove in the setting the following estimate:
\begin{proposition}
\label{shifted carleman}
There exists an $h_0>0$ such that for all $h\in(0,h_0)$ and $u\in C_0^\infty(\dot M^D)$ we have
\begin{equation}
\begin{gathered}
\|e^{-\varphi_{\epsilon}/h}h^2 \Delta_g e^{\varphi_\eps/h} u\|_{H^{-1}_{scl}(M^D)} \geq \frac{Ch}{\eps}\big( \sqrt{h} \|u\| + \| d\varphi u\| + \| d\varphi_\eps u\|+ \|h du\|_{H^{-1}_{scl}(M^D))}\big) 
\end{gathered}
\end{equation}
\end{proposition}
\noindent{\bf Proof.} By Lemma 3.2 of \cite{surface full data} one has the $L^2$ Carleman estimate
\[ \| e^{-\varphi_\eps/h}h^2\Delta_g e^{\varphi_\eps/h} u \| \geq \frac{Ch}{\eps} \Big( \sqrt{h} \|u\| + \| d\varphi u\| + \| d\varphi_\eps u\|+ \|h du\|\Big)\]
for all $u \in C_0^\infty(M^D_\delta)$. Now let $\chi\in C^\infty_0(M^D_\delta)$ be a cutoff so that $\chi =1$ on $\dot M^D$ and apply the above inequality to $\chi \langle hD\rangle^{-1} u$ for $u\in C^\infty_0(\dot M^D)$ where $ \langle hD\rangle^{-1}$ is the elliptic semiclassical pseudodifferential operator obtained by quantizing the symbol $\langle \xi \rangle ^{-1} := (1 + |\xi|^2_g)^{-1/2} \in S^{-1}(T^*M^D)$. Standard commutator calculus yields that
\begin{eqnarray}\label{L2 inequality}\nonumber
\| e^{-\varphi_\eps/h}h^2\Delta_g e^{\varphi_\eps/h} u \|_{H^{-1}_{scl}(M)} + \|e^{-\varphi_\eps/h}[h^2 \Delta_g, \chi] e^{\varphi_\eps/h}\langle hD\rangle^{-1}u\|+ \| \chi [e^{-\varphi_\eps/h}h^2\Delta_ge^{\varphi_\eps/h}, \langle hD\rangle^{-1}] u\|\\ \geq \frac{Ch}{\eps}\big(\sqrt{h} \|u\| + \| u d\phi_{\eps}\| + \| u d\phi\| + \| hdu\|_{H^{-1}_{scl}(M^D)}\big)
\end{eqnarray}
We compute the second term directly to obtain 
\[\|e^{-\varphi_\eps/h}[h^2 \Delta_g, \chi] e^{\varphi_\eps/h}\langle hD\rangle^{-1}u\| \leq h^2 \|u\| + h \|ud\varphi_{\eps}\| + h\|hdu\|_{H^{-1}_{scl}(M^D)}\]
and see that it can therefore be absorbed into the right side of inequality \eqref{L2 inequality}. Similarly if we write $e^{-\varphi_\eps/h}h^2\Delta_ge^{\varphi_\eps/h} = A + iB$ where \[Au = h^2 \Delta_gu - |d\varphi_\eps|^2u,\ \ iBu = {\rm div}_g (ud\varphi_\eps) + \langle d\varphi_\eps, du\rangle,\] we see that the third term on the left side of \eqref{L2 inequality} can be written as
\[[e^{-\varphi_\eps/h}h^2\Delta_ge^{\varphi_\eps/h}, \langle hD\rangle^{-1}] =h Op_h(\{a + ib, \langle \xi\rangle^{-1}\}) + h^2 Op_h(S^{-1}(T^*M^D))\]
which leads to the estimate
\[ \| \chi[A+iB,\langle hD\rangle^{-1}] u\| \leq h \|hdu\|_{H^{-1}_{scl}(M^D)} + h^2 \| u\|\]
and therefore can again be absorbed into the right side of inequality \eqref{L2 inequality}.


\qed

\end{subsection}

\begin{subsection}{Reflection Argument}
In this section we apply a reflection argument to prove Proposition \ref{shifted carleman estimate with boundary}. We first prove the estimate for the special case when $X=V=0$.
\begin{lemma}
\label{shifted CE with convexified weight}
For all $u\in C^\infty_0(M_0')$ we have that 
\[ \|e^{-\varphi_\epsilon/h} h^2\Delta_g e^{\varphi_\epsilon/h}u\|_{H^{-1}_{scl}(M)} \geq C \frac{h}{\epsilon} ( \sqrt{h} \|u\| + \|d\varphi  u\| + \|d\varphi_\epsilon  u\|+ \|hd \tilde u\|_{H^{-1}_{scl}(M^D)}).\] 
\end{lemma}
\noindent{\bf Proof.}

If $u$ is an element of $C^\infty_0(M_0')$, let $\tilde u$ denote its odd reflection which is an element of $C^\infty_0(\dot M^D)$ which extends trivially to a smooth odd function on $M^D$. We can now apply Lemma \ref{shifted carleman} to the compactly supported function $\tilde u \in C_0^\infty(\dot M^D)$ to obtain
\[ \|e^{-\varphi_\epsilon/h} h^2\Delta_g e^{\varphi_\epsilon/h} \tilde u\|_{H^{-1}_{scl}(M^D)} \geq C \frac{h}{\epsilon} ( \sqrt{h} \|\tilde u\| + \|d\varphi \tilde u\| + \|d\varphi_\epsilon \tilde u\|+ \|hd\tilde u\|_{H^{-1}_{scl}(M^D)})\] 

We now would like to use the symmetry of $\tilde u$ with respect to the pull-back by $R$ to argue that this estimate is comparable to the analogous one on $M$. This can be done with the help of the following
\begin{lemma}
\label{reflected norm}
Let $\tilde u\in C^\infty(M^D)$ be an odd function with respect to the involution $R$, that is, $R^*\tilde u = -\tilde u$, then 
\[\|\tilde u\|_{H^{-1}_{scl}(M^D)} = \sqrt{2} \|\tilde u\|_{H^{-1}_{scl}(M )}.\]

\end{lemma}
Indeed, since $\tilde u$ is odd and $\varphi_\epsilon$ is even we have that $e^{-\varphi_\epsilon/h} h^2\Delta_g e^{\varphi_\epsilon/h} \tilde u$ is also a smooth odd function on $M^D$. Thus we can apply Lemma \ref {reflected norm} to $e^{-\varphi_\epsilon/h} h^2\Delta_g e^{\varphi_\epsilon/h} \tilde u$ to obtain
\begin{eqnarray*}
\|e^{-\varphi_\epsilon/h} h^2\Delta_g e^{\varphi_\epsilon/h}u\|_{H^{-1}_{scl}(M)} &\geq& C \frac{h}{\epsilon} ( \sqrt{h} \|u\| + \|d\varphi  u\| + \|d\varphi_\epsilon  u\|+ \|hd \tilde u\|_{H^{-1}_{scl}(M^D)})\\
 &\geq& C \frac{h}{\epsilon} ( \sqrt{h} \|u\| + \|d\varphi  u\| + \|d\varphi_\epsilon  u\|+ \|hd  u\|_{H^{-1}_{scl}(M)})
 \end{eqnarray*}
\qed\\
We complete this subsection we must provide\\
\noindent{\bf Proof of Lemma \ref{reflected norm}.} We compute directly the $H^{-1}_{scl}(M^D)$ norm of $u\in C^\infty(M^D)$.
\[ \|u\|_{H^{-1}_{scl}(M^D)} := \sup\limits_{ v \in H^1(M^D),\ \|v\|_{H^{1}_{scl}(M^D)}\leq 1}\langle u, v\rangle = \int_{M^D} u \hat v  \]
where $\hat v$ is the unique maximizer in $\hat v \in H^1(M^D)$ with $\|\hat v\|_{H^{1}_{scl}(M^D)}= 1$. We decompose $\hat v$ into its odd and even parts by writing 
\[\hat v (x,y) = \frac{\hat v + R^*\hat v}{2} + \frac{\hat v - R^*\hat v}{2} := \hat v^+ +\hat v^-.\]  
Observe that since $u$ is odd by assumption we have 
\[\int_{M^D} u \hat v ^+  = \int_{M^D} R^* u R^* \hat v^+   = -\int_{M^D} u \hat v ^+  ,\ \int_{M^D} u \hat v ^-  = 2\int_{M} u  \hat v^-   \]
and thus we can write 
\begin{eqnarray}
\label{twice the integral}
 \|u\|_{H^{-1}_{scl}(M^D)} = \int_{M^D} u \hat v = \int_{M^D} u \hat v^-  = 2\int_M u\hat v^-  .\end{eqnarray}

Note that since $\int_{M^D} \hat v^+ \overline{\hat v^-}  = 0$ and $\int_{M^D} \langle d \hat v^+, d\hat v^-\rangle   =  \int_{M^D}\Delta\hat v^+  \overline{\hat v^-}  = 0$, we can write the $H^1_{scl}(M^D)$ norm of $\hat v$ as
\[1=\|\hat v\|_{H^{1}_{scl}(M^D)}^2 = \|\hat v^+\|_{H^{1}_{scl}(M^D)}^2 + \|\hat v^-\|_{H^{1}_{scl}(M^D)}^2.\]
From this we can conclude that $\hat v ^-$ is in the unit ball of $H^1_{scl}(M^D)$ and by the uniqueness of maximizer we have that $\hat v^- = \hat v$. Furthermore, since $\hat v^-$ is odd, it vanishes along the fixed points of the involution $R$. As the involution $R$ fixes the boundary $\partial M$, this means that $v^-\mid_{\partial M} = 0$ and therefore $\hat v^-\mid_{M}\in H^1_0(M)$ with semiclassical norm $\|\hat v^-\|_{H^1_{scl}(M)} = \frac{1}{\sqrt{2}}$. So by \eqref{twice the integral} we have that
\[ \|u\|_{H^{-1}_{scl}(M^D)}  = 2 \int_{M} u \hat v^-  \leq 2\sup\limits_{ v\in H^1_0(M), \|v\|_{H^1_{scl}(M)}\leq \frac{1}{\sqrt 2}} \int_{M} uv  = \sqrt{2}\|u\|_{H^{-1}_{scl}(M)}\]

This inequality goes the other direction by observing that for odd functions $u \in C^\infty(M^D)$ we have
\[ \|u\|_{H^{-1}_{scl}(M^D)} \geq \sup\limits_{ v\in H^1_0(M), \|v\|_{H^1_{scl}(M)}\leq \frac{1}{\sqrt 2}} \int_{M} uv +\sup\limits_{ v\in H^1_0(M), \|v\|_{H^1_{scl}(M)}\leq \frac{1}{\sqrt 2}} \int_{R(M)} uR^*v = \sqrt{2} \|u\|_{H^{-1}_{scl}(M)}.\]
\qed
\end{subsection}
\begin{subsection}{Proof of Proposition \ref{shifted carleman estimate with boundary}.}
By Lemma \ref{shifted CE with convexified weight} we have for $u\in C^\infty_0(M_0')$ the estimate for the Laplacian with convexified weights:
\[ \|e^{-\varphi_\epsilon/h} h^2\Delta_g e^{\varphi_\epsilon/h}u\|_{H^{-1}_{scl}(M)} \geq C \frac{h}{\epsilon} ( \sqrt{h} \|u\| + \|d\varphi  u\| + \|d\varphi_\epsilon  u\|+ \|hd  u\|_{H^{-1}_{scl}(M)}).\] 
If we replace $\Delta_g$ by the operator $L_{X,V} := (d+iX)^*(d+iX) + V$ we will obtain errors on the left side:
{\small\begin{eqnarray*}  &&\|e^{-\varphi_\epsilon/h}h^2 L_{X,V} e^{\varphi_\epsilon/h} u\|_{H^{-1}_{scl}(M)}+h^2 \|Qu\|+ h\| \langle d\varphi_\eps, X\rangle u\|+ h\| \langle X, hdu\rangle\|_{H^{-1}_{scl}}\geq\\ && \|e^{-\varphi_\epsilon/h} h^2\Delta_g e^{\varphi_\epsilon/h}u\|_{H^{-1}_{scl}(M)}  \geq C\frac{h}{\epsilon}( \sqrt{h} \|u\| + \| d\varphi u\| + \|d\varphi_{\epsilon}u \| + \|h du\|_{H^{-1}_{scl}(M))}\end{eqnarray*}
}
for some $Q\in L^\infty$. Since $X\in W^{1,\infty}(M)$ and $Q\in L^\infty(M)$ all the errors on the left side can be absorbed into the right side of the inequality. We now replace $u$ in the above estimate by $e^{\frac{1}{2\epsilon} \sum \limits_{j=1}^N \varphi_j^2} u$ so that $e^{\varphi_\epsilon/h} e^{\frac{1}{2\epsilon} \sum \limits_{j=1}^N \varphi_j^2} u = e^{\varphi/h} u$ and the estimate follows.

\qed

\end{subsection}

\end{section}


\begin{section}{Boundary Determination}
We begin the section by stating the local boundary determination result. The statement was proven in the Euclidean case by \cite{BrSa} and \cite{salothesis}. A slight generalization to the case of Riemann surfaces was done in \cite{gafa}. The results are statement for the global Dirichlet-Neumann map but as the methods are local they can be generalized without modification to show
\begin{proposition}
\label{boundary value determination}
Let $X_1, X_2\in W^{3,p}(M; T^*M_0)$ be real-valued 1-forms and $V_1, V_2 \in W^{2,p}(M_0)$ be functions. If assumptions \eqref{main assumption} and \eqref{boundary normal assumption} are satisfied then $\iota^*_{\partial M_0}  X_1\mid_{\partial M_0\backslash \Gamma} = \iota^*_{\partial M_0}X_2 \mid_{\partial M_0\backslash \Gamma}$. 
\end{proposition}

An immediate consequence of this is the following. If $M$ is a surface containing $M_0$ such that $\Gamma\subset \partial M_0\cap \partial M$ and $M\backslash M_0$ is simply connected, then there exists $W^{1,\infty}(M)$ and $L^\infty(M)$ extensions of $X_j$ and $V_j$ respectively such that $X_1\mid_{M\backslash M_0} = X_2\mid_{M\backslash M_0}$, $\langle \nu, X_1 - X_2\rangle\mid_{\partial M}=0$, and $V_1\mid_{M\backslash M_0} = V_2\mid_{M\backslash M_0}$. Furthermore, on the surface $M$ the Cauchy data for the extended coefficients, which we still denote by $X_j$ and $V_j$, satisfy ${\cal C}_{X_1,V_1,\partial M\backslash \Gamma}={\cal C}_{X_2,V_2,\partial M\backslash \Gamma}$.

Observe that if one multiplies the metric $g$ by a conformal factor $e^K$, the above relation for the Cauchy data holds for $V_j$ replaced by $e^{-K}V_j$. As such we may assume without loss of generality that for each connected component of $\partial M$ there exists an interior neighbourhood which is isometric to the flat cylinder $[0,\epsilon]\times S^1$ with metric $dt^2+d\theta ^2 $ (\cite{mazzeo taylor}). Furthermore, if $q_0\in  \partial M\backslash \bar\Gamma$ and $\delta >0$ are chosen so that $M_0$ is contained in $M_0' := M\backslash \bar B_\delta(q_0)$ with $B_\delta(q_0) := \{p\in M^D\mid d(p,q_0) <\delta\}$, then on the surface $M_0'$ one again has ${\cal C}_{X_1,V_1,\partial M_0'\backslash \Gamma'}={\cal C}_{X_2,V_2,\partial M_0'\backslash \Gamma'}$ and $\langle \nu, X_1 - X_2\rangle\mid_{\partial M_0'}=0$. Here $\Gamma':= \partial M_0' \cap \partial M$ contains $\bar\Gamma$. We summarize this discussion in the following 
\begin{coro}
\label{same cauchy data in M0'}
Let $M$ and $M_0'$ be the surfaces defined above. There exists $W^{1,\infty}(M)$ and $L^\infty(M)$ extensions to the coefficients $X_j$ and $V_j$ respectively such that on $M$ one has $X_1\mid_{M\backslash M_0} = X_2\mid_{M\backslash M_0}$, $V_1\mid_{M\backslash M_0} = V_2\mid_{M\backslash M_0}$. On the surface $M_0'$ one has $\langle \nu, X_1 - X_2\rangle \mid_{\partial M_0'} = 0$ and the Cauchy data satisfies ${\cal C}_{X_1,V_1,\partial M_0'\backslash \Gamma}={\cal C}_{X_2,V_2,\partial M_0'\backslash \Gamma}$.
\end{coro}
The advantage in working with $M_0'\subset M$ with flat cylindrical metric near $\partial M$ is that its double as a subset of $M^D$ with metric given by $R*g = g$ is a manifold with both smooth metric and boundary. 

\begin{subsection}{Boundary Values of $F_{A_j}$}
Let $M$ and $M'_0$ be the surface constructed in the previous section. Let $X\in W^{1,\infty}(M, T^*M)$ be a real-valued 1-form on $M$ which can be decomposed into its $T^*_{0,1} M$ and $T^*_{1,0} M$ component which we denote by $A$ and $\bar A$ respectively. If $\Gamma':= \partial M_0' \cap \partial M$, Proposition \ref{surject} asserts that for all $p\in (1,\infty)$ one can find $\alpha\in W^{2,p}(M)$ which is real-valued along $\Gamma'$ solving $\bar\partial \alpha = A$ so that 
\[\bar\partial e^{i\alpha} = i e^{i\alpha} A\ \ \text{in}\ \ \ M_0',\ \ \ \ |e^{i\alpha}|\mid_{\Gamma'} = 1.\]

Of course, $e^{i\alpha}$ is not the unique non-vanishing solution to this boundary value problem. Indeed, one can multiply $e^{i\alpha}$ by any non-vanishing holomorphic function which is unitary along $\Gamma'$ to obtain another solution. It turns out the solutions of these boundary value problems are closely related to the Cauchy data of $L_{X,V}$.
\begin{proposition}
\label{boundary holomorphic}
If $X_j\in W^{1,\infty}(M, T^*M)$ are real valued 1-forms and $V_j\in L^{\infty}(M)$ for $j=1,2$ satisfy $\langle \nu, X_1 - X_2\rangle\mid_{\partial M} = 0$ and for $p\in (1,\infty)$ large, let $\alpha_j \in W^{2,p}(M)$ be a solution of 
\begin{eqnarray}
\label{alpha equation}
\bar\partial \alpha_j = A_j,\ \ \ \ \alpha_j\mid_{\Gamma'}\in \R.
\end{eqnarray}
Suppose $\cal{C}_{X_1, V_1, \partial M_0'\backslash\Gamma'} = \cal{C}_{X_2, V_2, \partial M_0'\backslash\Gamma}$, then\\
i) $e^{i(\alpha_1 -\alpha_2)}\mid_{\partial M_0' \backslash \Gamma'}$ extends to a non-vanishing holomorphic function $\Psi$ on $M_0'$ which is unitary along $\Gamma'$. Furthermore, $\Psi\mid_{M_0} \in C^\infty(M_0)$ up to the boundary.\\
ii) $e^{i(\bar\alpha_1 - \bar\alpha_2)}\mid_{\partial M_0' \backslash \Gamma'}$ extends to a non-vanishing antiholomorphic function $\Psi$ on $M_0'$ which is unitary along $\Gamma'$. Furthermore, $\Psi\mid_{M_0} \in C^\infty(M_0)$ up to the boundary.\\ 


\end{proposition}

\noindent{\bf Proof.} Since (ii) and (i) are equivalent we will only prove (ii).


Since $\alpha_j\mid_{\Gamma'}\in \R$, one can define a Lipschitz piece-wise smooth function $F_{1,2}\in W^{1,\infty}(\dot M^D)$ on $\dot M^D\cup \partial \dot M^D$ by
\begin{eqnarray}
\label{def of F12}
F_{1,2} =
\left\{
	\begin{array}{ll}
		 e^{i(\bar\alpha_2 - \bar\alpha_1)} & \mbox{on } M_0' \\
		R^*e^{i(\alpha_2 - \alpha_1)} & \mbox{on } R(M_0'). 
	\end{array}
\right.
\end{eqnarray}
In fact one can show that $F_{1,2}\in W^{2,p}(\dot M^D)$. Indeed, since $\im(\alpha_1 -\alpha_2)$ vanishes along $\Gamma'$ by assumption, its odd extension across $\Gamma '$ is an element of $W^{2,p}(\dot M^D)$. To show that $F_{1,2}\in W^{2,p}(\dot M^D)$ we need to check that the even extension across $\Gamma'$ of $\re(\alpha_1 - \alpha_2)$ has two derivatives as well. This is equivalent to showing that $\pl_\nu \re(\alpha_1 - \alpha_2)$ vanishes along $\Gamma'$. This can be done by using $\langle\nu, X_1 - X_2\rangle = 0$ along $\partial M$ (Corollary \ref{same cauchy data in M0'}
) and the fact that 
\[0=\langle \nu,X_1 - X_2 \rangle= \langle \nu,(A_1- A_2) + (\bar A_1 - \bar A_2) \rangle= \langle \nu,\bar\pl (\alpha_1 - \alpha_2) + \pl(\bar\alpha_1 - \bar\alpha_2)\rangle.\]
Writing this out in boundary normal coordinates yields that $\pl_\nu \re(\alpha_1 - \alpha_2)=0$ along $\Gamma'$ and thus  $F_{1,2}\in W^{2,p}(\dot M^D)$.

We have the following Lemma for the boundary value of $F_{1,2}\mid_{\partial \dot M^D}\in W^{2-\frac{1}{p},p}(\partial \dot M^D)$ defined by \eqref{def of F12}:
\begin{lemma}
\label{antimeromorphic statement}
The function $F_{1,2}\mid_{\partial \dot M^D}$ has an antiholomorphic extension $\Psi$ into the surface $\dot M^D$. 
\end{lemma}


Assuming Lemma \ref{antimeromorphic statement}, we need to show that $\Psi$ is non-vanishing. To this end we switch the indices $1$ and $2$ in \eqref{def of F12} to show that $F_{2,1}\mid_{\partial \dot M^D} = F_{1,2}^{-1}\mid_{\partial \dot M^D}$ is the boundary value of an antiholomorphic function on $\dot M^D$. By uniqueness, this antiholomorphic function must be $\Psi^{-1}$ and we have that $\Psi$ is non-vanishing.

We now show that $\Psi\mid_\Gamma$ is unitary. To this end, observe that $F_{1,2}\mid_{\partial \dot M^D}$ satisfies the symmetry condition $\overline{(R^*F_{1,2})}^{-1}\mid_{\partial \dot M^D} = F_{1,2}\mid_{\partial\dot M^D}$. By uniqueness this implies that the antiholomorphic function $\overline{(R^*\Psi)}^{-1}$ is identical to $\Psi$. As such, since $R$ is the identity on $\Gamma'$, we have that $\overline{(\Psi)}^{-1}(p) = \Psi(p)$ for all $p\in \Gamma'$; that is, $\Psi\mid_{\Gamma'}$ is unitary. Restricting the function $\Psi$ to $M_0'$ we have the desired antiholomoprhic extension to $e^{i(\bar\alpha_1 - \bar\alpha_2)}\mid_{\partial M_0' \backslash \Gamma'}$. The smoothness of $\Psi$ on the closure of $M_0$ follows from the fact that $M_0$ is compactly contained in $\dot M^D$.\qed\\ 
An immediate consequence of Proposition \ref{boundary holomorphic} is the following
\begin{coro}
\label{same boundary condition for conjugation factor}
There exists an open subset $\Gamma_0\subset \partial M_0$ containing $\Gamma$ whose complement $\partial M_0\backslash \Gamma_0$ contains an open subset such that for all $p\in (1,\infty)$ one can choose solutions $F_{A_j}, F_{\bar A_j}\in W^{2,p}(M_0)\cap W^{4,p}_{loc}(M_0)$ solving
\begin{eqnarray}
\label{conjugating factor with bc}
\bar\pl F_{A_j} = iA_j F_{A_j}\ \ \text{in}\ \ M_0,\ \ \ \ |F_{A_j}|\mid_{\Gamma_0} = 1
\end{eqnarray}
and 
\begin{eqnarray}
\label{antiholo conjugating factor with bc}
\pl F_{\bar A_j} = i\bar A_j F_{\bar A_j}\ \ \text{in}\ \ M_0,\ \ \ \ |F_{\bar A_j}|\mid_{\Gamma_0} = 1
\end{eqnarray}
such that $F_{A_1} \mid_{\partial M_0\backslash \Gamma_0}= F_{A_2}\mid_{\partial M_0\backslash \Gamma_0}$ and $F_{\bar A_1}\mid_{\partial M_0\backslash \Gamma_0}= F_{\bar A_2}\mid_{\partial M_0\backslash \Gamma_0}$.
\end{coro}
\noindent{\bf Proof. }We will only prove the statement for $F_{A_j}$ as the one for $F_{\bar A_j}$ can be achieved by the same argument. Let $M$ be a surface with boundary containing $M_0$ such that $\Gamma\subset \partial M_0\cap\partial M$. Define $M_0'$ by removing a small half-disk around boundary point $q_0\in \partial M \backslash \bar\Gamma$ such that $M_0 \subset M_0'$ and $\Gamma' := \partial M_0'\cap \partial M$ compactly contains $\Gamma$.

By Corollary \ref{same cauchy data in M0'} there exists $W^{1,\infty}(M)$ extensions of $X_j$ and $V_j$ respectively such that ${\cal C}_{X_1,V_1, \partial M_0'\backslash \Gamma'} ={\cal C}_{X_2,V_2, \partial M_0'\backslash \Gamma'}$, $X_1 = X_2$ on $M\backslash M_0$, and $\langle \nu, X_1-X_2\rangle\mid_{\partial M_0'}=0$. Lemma \ref{surject} shows that for all $p\in (1,\infty)$ if denotes $A_j := \pi_{0,1} X_j$ then there exists $\alpha_j\in W^{2,p}(M)$ solving
\[\bar\partial \alpha_j = A_j,\ \ \ \ \alpha_j\mid_{\Gamma'}\in \R.\]
Observe that since $X_j\mid_{M_0} \in W^{3,\infty}(M_0)$ elliptic regularity stipulates that $\alpha_j\in W^{4,p}_{loc}(M_0)$ for all $p\in (1,\infty)$. Proposition \ref{boundary holomorphic} asserts that the boundary value $e^{i(\alpha_1 - \alpha_2)}\mid_{\partial M'_0 \backslash\Gamma'}$ extends to a non-vanishing holomorphic function $\Psi$ on $M_0'$ which is unitary along $\Gamma'$ and smooth on the closure of $M_0$. 

Setting $F_{A_1} := e^{i\alpha_1}$ and $F_{A_2} := \Psi e^{i\alpha_2}$ one has that 
\[\bar\partial F_{A_j} = F_{A_j} A_j\ \ \text{in}\ \ M_0', \ \ F_{A_1} = F_{A_2}\ \ \text{on}\ \ \partial M_0'\backslash \Gamma',\ \ |F_{A_j}| = 1\ \ \text{on}\ \ \Gamma'.\]
Furthermore, using the fact that $X_1 = X_2$ in $M_0'\backslash M_0$ one sees that $F_{A_1} F_{A_2}^{-1}$ is holomorphic in $M_0'\backslash M_0$. The boundary condition $F_{A_1} = F_{A_2}$ on $\partial M_0'\backslash\Gamma'$ forces $F_{A_1} = F_{A_2}$ in $M_0'\backslash M_0$ and therefore if one defines $\Gamma_0:= \partial M_0 \cap \partial M_0'\subset \Gamma'$ one has $F_{A_1} = F_{A_2}$ on $\partial M_0\backslash \Gamma_0$ and $|F_{A_j}|= 1$ on $\Gamma_0$. 
\qed

It remains to prove Lemma \ref{antimeromorphic statement} and it is the goal of the next subsection.
\end{subsection}

\begin{subsection}{\bf Proof of Lemma \ref{antimeromorphic statement}.}

The strategy which we will follow is to use the equivalence of the Cauchy data ${\cal C}_{X_1,V_1,\partial M_0'\backslash \Gamma}={\cal C}_{X_2,V_2,\partial M_0'\backslash \Gamma}$ on $M_0'$ to derive an orthogonality condition similar to the one in Lemma \ref{perpendicular to all conjugate even} on the double $\dot M^D$. This will be done through the standard boundary integral identity, assuming that ${\cal C}_{X_1,V_1,\partial M_0'\backslash \Gamma}={\cal C}_{X_2,V_2,\partial M_0'\backslash \Gamma}$ on $M_0'$,
{\small \begin{eqnarray}
\label{boundary integral id in boundary holomorphic}
0&=&\int_{M_0'} \bar u_2 (L_{X_2,V_2} - L_{X_1,V_1}) u_1 {\rm dvol}_g \\\nonumber &=&\int_{M_0'} \bar u_2 (A_1 - A_2) \wedge \partial u_1 - \bar u_2(\bar A_1 -\bar A_2)\wedge \bar\partial u_1 + \bar u_2 (V_1 -V_2) u_1{\rm dvol}_g
\end{eqnarray}}
for all solutions $u_j$ of $L_{X_j, V_j}u_j  = 0$ on $M_0'$ and vanishing on $\Gamma'\subset M_0'$.

Let $\Phi$ be the Morse holomorphic function on $M$ given by Proposition \ref{criticalpoints} which is real valued along $\Gamma'$. If $\{p_1,.., p_N\}$ are critical points of $\Phi$ in $M_0'\cup \partial M_0'$, we consider the set of antiholomorphic 1-forms $b\in W^{2,\infty}(M_0', T_{1,0}^*M_0')$satisfying 
\begin{eqnarray}
\label{class of b}
\iota_{\partial M_0}^*b \mid_{\Gamma'} \in \R,\ \ \ b(p_j) = 0\ \ \text{to }k\text{-th order for }j= 1,..N
\end{eqnarray}
For all such $b$, $\Phi$ and $\alpha_j$ satisfying \eqref{alpha equation} the ansatz given by 
\begin{eqnarray}
\label{u0 ansatz}
u_0 :=e^{\bar\Phi/h} he^{-i\bar\alpha_1}\frac{ b}{\bar\pl \bar\Phi}-e^{\Phi/h} he^{-i\alpha_1} \frac{ \bar b}{\pl \Phi}
\end{eqnarray}
vanishes along $\Gamma'$. Here we denote by $\frac{ b}{\bar\pl \bar\Phi}$ the unique function satisfying $ \frac{ b}{\bar\pl \bar\Phi}\bar\pl \bar\Phi = b$. Since $b$ vanishes to $k$-th order at all critical points of $\Phi$, this function is an element of $W^{2,\infty}(M_0')$.

Writing $L_{X,V} = (d+iX)^*(d+iX) +V$ as 
\begin{eqnarray}
\label{LXV in in dbar form}
L_{X,V} = -2i \star e^{-i\bar\alpha} \partial |e^{i\alpha}|^{-2} \bar\partial e^{i\alpha} +  Q = -2i\star e^{-i\alpha}\bar\partial |e^{i\alpha}|^2 \partial e^{i\bar\alpha} + \tilde Q
\end{eqnarray}
for some $Q, \tilde Q\in L^{\infty}(M_0')$, one sees that the ansatz $u_0$ satisfies
\[e^{-\phi/h}L_{X_1,V_1} u_0 =  O_{L^\infty}(h),\ \ \ u_0\mid_{\Gamma'} = 0.\]
To obtain a solution one then applies Corollary \ref{H1 solvability} to obtain $u_1$ solving $L_1 u_1=0$ of the form
\begin{eqnarray}
\label{u1}
u_1 = u_0 + e^{\phi/h} r,\ \ u_1\mid_{\Gamma'} = 0,\ \ \|r_1\| + \|hdr_1\| \leq Ch\sqrt{h}.
\end{eqnarray}
Using \eqref{LXV in in dbar form} again we can also directly show that
\[e^{\phi/h}L_{X,V} (e^{-\Phi/h}e^{-i\alpha} - e^{-\bar\Phi/h}e^{-i\bar\alpha}) = O_{L^\infty}(1),\ \ \  (e^{-\Phi/h}e^{-i\alpha} - e^{-\bar\Phi/h}e^{-i\bar\alpha} )\mid_{\Gamma'} = 0.\ \ \] 
Therefore, by applying Corollary \ref{H1 solvability} again we obtain solutions $u_2$ to $L_{X_2,V_2} u_2 = 0$ of the form
\begin{eqnarray}
\label{u2}
u_2 = (e^{-\Phi/h}e^{-i\alpha_2}- e^{-\bar\Phi/h}e^{-i\bar\alpha_2} ) + e^{-\phi/h} r_2,\ \ \ u_2\mid_{\Gamma'} =0,\ \ \ \|r_2\| \leq C\sqrt{h}.
\end{eqnarray}
Simple computation from expression \eqref{u1} yields that 
\[\partial u_1= -e^{\Phi/h} e^{-i\alpha_1}\bar b + e^{\phi/h}O_{L^2}(\sqrt{h})  ,\ \ \bar\partial u_1 = e^{\bar\Phi/h} e^{-i\bar\alpha_1} b + e^{\phi/h} O_{L^2}(\sqrt{h}).\]
Combining this with the expression \eqref{u2} and plug them into \eqref{boundary integral id in boundary holomorphic} we obtain
\[0 = \int_{M_0}  e^{i(\alpha_2-\alpha_1)}  (A_1 - A_2) \wedge\bar b - e^{i(\bar\alpha_2-\bar\alpha_1)}(\bar A_1 -\bar A_2)\wedge b   + o(1).\]
Using $\bar\partial e^{i\alpha_j} = i e^{i\alpha_j} A_j$, $\partial e^{-i\bar\alpha_j} = - ie^{-i\bar\alpha_j}\bar A_j$, and $\partial b =0$ we obtain in the limit $h\to 0$,
\[0 = \int_{M_0'} \partial (e^{i(\bar\alpha_2-\bar\alpha_1)} b ) -\bar\partial (e^{i(\alpha_2 - \alpha_1)}\bar b) = \int_{\partial M_0'} e^{i(\bar\alpha_2-\bar\alpha_1)} \iota^*_{\partial M_0} b  -  e^{i(\alpha_2 - \alpha_1)}\iota^*_{\partial M_0}\bar b.\]
The antiholomorphic 1-form $b$ satisfies the boundary condition given in \eqref{class of b} so that \[\iota_{\partial M_0'}^*b - \iota_{\partial M_0'}^*\bar b = 0\ \  {\rm on}\ \ \Gamma'\] and $\alpha_j\mid_{\Gamma'}\in \R$ on $\Gamma$ by \eqref{alpha equation}. Therefore the integrand in above boundary integral identity vanishes on $\Gamma$ to give
\begin{eqnarray}
\label{two parts of boundary holomorphic}
0=\int_{\partial M_0'\backslash\Gamma'} (\bar F_{A_2})^{-1}\bar F_{A_1}\iota^*_{\partial M_0'} b  -  F_{A_2} F_{A_1}^{-1} \iota^*_{\partial M_0'}\bar b.
\end{eqnarray} 
for all antiholomorphic 1-form $b$ satisfying \eqref{class of b}.

Note that since $\iota_{\partial M_0'}^*b \mid_{\Gamma'} \in \R$, the antiholomorphic 1-form $b$ on $M_0'$ extends to a conjugate even antiholomorphic 1-form $\eta$ on $ \dot M^D $. 
Expressed in the antiholomorphic 1-form $\eta$ and the function $F_{1,2}$ defined in \eqref{def of F12}, the integral in \eqref{two parts of boundary holomorphic} can be written as an integral along $S^1 = \partial \dot M^D$ to give
\[ 0=\int_{\partial M_0'\backslash\Gamma'}F_{1,2}\iota^*_{\partial M_0'} \eta  +\int_{R(\partial M_0' \backslash \Gamma')}  F_{1,2} \iota^*_{\partial M_0'}\eta = \int_{ \partial \dot M^D} F_{1,2} \iota^*_{\partial \dot M^D}\eta.\]
As $b$ vary over the space of antiholomorphic 1-forms on $M_0'$ satisfying \eqref{class of b}, its conjugate even extension $\eta$ vary over the space of all conjugate even antiholomorphic 1-forms on $\dot M^D$ vanishing at $\{p_1,..,p_N, R(p_1),.. R(p_N)\}$. Therefore, by Lemma \ref{perpendicular to all conjugate even}, the function $F_{1,2}\mid_{\partial \dot M^D}$ is the boundary value of an antimeromorphic function $\Psi$ on $\dot M^D$ with poles at \[\{p_1,..,p_N, R(p_1),.. R(p_N)\} \cap \dot M^D.\]

We would like to show that the antimeromorphic extension $\Psi$ is actually antiholomorphic by showing that all poles are removable. To this end construct by Lemma \ref{control the zero} a holomorphic function $\tilde \Phi$ on $M$ which is real valued along $\Gamma$ such that $p_1$ is not a critical point of $\tilde \Phi$. We can then use the perturbation argument of Lemma \ref{morsedense} to ensure that it is Morse. By applying the same argument with $\tilde \Phi$ in place of $\Phi$ we can assert that $F_{1,2}\mid_{\partial \dot M^D}$ extends to a antimeromorphic function $\tilde\Psi$ for which $p_1$ and $R(p_1)$ are not poles. By uniqueness $\Psi$ and $\tilde \Psi$ are identical since they have the same boundary value. Therefore we can conclude that $\Psi$ has a removable singularity at $p_1$ and $R(p_1)$. Applying the same argument for the other points we have that $\Psi$ is antiholomorphic.

\qed

\end{subsection}

\begin{subsection}{Proof of Proposition \ref{new boundary integral id}}
An immediate consequence of Proposition \ref{boundary holomorphic} is the new boundary integral identity of Proposition \ref{new boundary integral id} which is more convenient for recovering information about first-order coefficients. Let $F_{A_1}$ and $F_{A_2}$ be non-vanishing functions solving \eqref{conjugating factor with bc} and by Corollary \ref{same boundary condition for conjugation factor} we can choose them to satisfy $F_{A_1} = F_{A_2}$ on the line segment $\partial M_0\backslash \Gamma_0$ for some $\Gamma_0\subset\partial M_0$ containing $\Gamma$. Similarly, Corollary \ref{same boundary condition for conjugation factor} allows one to make the analogues choice for $F_{\bar A_j}$ solving $\partial F_{\bar A_j} = iF_{\bar A_j} \bar A_j$. 

For these choices of $F_{A_j}$ and $F_{\bar A_j}$ we consider solutions to the boundary value problem for systems
\begin{eqnarray}
\label{conjugated system}
\begin{pmatrix}
  0 & \bar\pl^* \\
  \bar\pl&0
 \end{pmatrix}\begin{pmatrix}\tilde u_j \\ \tilde\omega_j\end{pmatrix}
+ \begin{pmatrix}
  v_j & 0 \\
  0& v_j'
 \end{pmatrix}\begin{pmatrix}\tilde u_j \\ \tilde\omega_j\end{pmatrix} = 0, \ \ \ \ \tilde u_j\mid_{\Gamma_0} = 0
\end{eqnarray}
where $v_j = \frac{1}{2} |F_{\bar A_j}|^2 Q_j$, $v'_j =- |F_{A_j}|^2$, and $Q_j = \star dX_j + V_j$. Setting $u_j = F_{A_j}^{-1} \tilde u_j$ and $\omega_j = F_{\bar A_j}^{-1} \tilde\omega_j$, system (\ref{conjugated system}) is equivalent to $(u_j, \omega_j)$ solving the system
\begin{eqnarray}
\label{first order system}
\begin{pmatrix}
  0 & (\bar\pl + i A_j)^* \\
  \bar\pl + i A_j&0
 \end{pmatrix}\begin{pmatrix}\ u_j \\ \omega_j\end{pmatrix}
+ \begin{pmatrix}
  Q_j& 0 \\
  0& -1
 \end{pmatrix}\begin{pmatrix} u_j \\ \omega_j\end{pmatrix} = 0,\ \ \ u_j\mid_{\Gamma_0} = 0
\end{eqnarray}
and this holds if and only if $L_{X_j, V_j} u_j = 0$. Consequently, if a pair $(\tilde u_1 , \tilde\omega_1)$ solves (\ref{conjugated system}) with $\tilde u_1 \mid_{\Gamma_0} = 0$, then by the fact that $\cal{C}_{V_1, X_1, {\partial M_0\backslash \Gamma}} = \cal{C}_{V_2, X_2, {\partial M_0\backslash \Gamma}}$, there exists a $u_2$ solving $L_{X_2,V_2}u_2 =0$ such that $(u_1\mid_{\pl M}, (d + iX_1)u_1\mid_{{\partial M_0\backslash \Gamma}}) =(u_2\mid_{\pl M}, (d + iX_2)u_2\mid_{{\partial M_0\backslash \Gamma}})$. By equation (\ref{first order system}) this means that $(u_1\mid_{\pl M}, \omega_1\mid_{{\partial M_0\backslash \Gamma}}) =(u_2\mid_{\pl M}, \omega_2\mid_{{\partial M_0\backslash \Gamma}})$. As we have chosen $F_{A_j}$ and $F_{\bar A_j}$ so that $F_{A_1} = F_{A_2}$ and $F_{\bar A_1} = F_{\bar A_2}$ on ${\partial M_0\backslash \Gamma_0}$, we conclude that $(\tilde u_1\mid_{\pl M}, \tilde\omega_1\mid_{{\partial M_0\backslash \Gamma_0}}) =(\tilde u_2\mid_{\pl M}, \tilde\omega_2\mid_{{\partial M_0\backslash \Gamma_0}})$. We therefore conclude that the systems (\ref{conjugated system}) for $j=1,2$ has the same partial Cauchy data 
\[\{(\tilde u_j, \tilde \omega_j\mid_{\partial M_0\backslash \Gamma_0}) \mid  \supp (\tilde u_j)\subset \partial M_0\backslash\Gamma_0, \begin{pmatrix}
  0 & \bar\pl^* \\
  \bar\pl&0
 \end{pmatrix}\begin{pmatrix}\tilde u_j \\ \tilde\omega_j\end{pmatrix}
+ \begin{pmatrix}
  v_j & 0 \\
  0& v_j'
 \end{pmatrix}\begin{pmatrix}\tilde u_j \\ \tilde\omega_j\end{pmatrix} = 0
\}\]

Standard boundary integral identity for first order systems then yields that for any two sets of solutions $ \begin{pmatrix}\tilde u_j \\ \tilde\omega_j\end{pmatrix} $,
\[ \int_{M_0}\langle \begin{pmatrix}\tilde u_2 \\ \tilde\omega_2\end{pmatrix},\begin{pmatrix}
  v_1- v_2 & 0 \\
  0& v_1'-v_2'
 \end{pmatrix}  \begin{pmatrix}\tilde u_1 \\ \tilde\omega_1\end{pmatrix} \rangle= 0\]
 provided that $\tilde u_1$ and $\tilde u_2$ vanishes on $ \Gamma_0$. The boundary integral identity \eqref{new boundary integral equation} follows by definition of $v_j$ and $v_j'$. \qed

\end{subsection}

\end{section}


\begin{section}{Construction of CGO - Part I}
In this section we construction complex geometrics solving $L_{X,V} u = 0$ which vanish on $\Gamma_0\subset \partial M_0$. The solutions we construct here will be inserted into boundary integral identity \eqref{new boundary integral equation} to show that $|F_{A_1}| = |F_{A_2}|$.

Let $\Phi$ be a holomorphic Morse function on $M_0$ which is real valued on $\Gamma_0$. Suppose $\{p_0,..,p_N\}$ are the critical points of $\Phi$ in $\bar {M_0}$ with $p_0$ in the interior. We apply Lemma \ref{control the zero} to construct antiholomorphic 1-form $b$ on $M$ smooth up to the boundary such that $b(p_j)= 0$ to $k$-th order at $p_1,..,p_N$ and $b(p_0)\neq0$. Let $F_A\in W^{2,p}(M_0) \cap W^{4,p}_{loc}(M_0)$ be a non-vanishing function for large $p\in (1,\infty)$ satisfying $\bar\partial F_A = iA F_A$ and $|F_A| = 1$ on $\Gamma_0$. We choose a smooth cut-off $\chi \in C^\infty_0(M_0)$ supported in a small neighbourhood of $p_0$ and define 
\begin{eqnarray}
\label{def of u0 and u0'}
u_{0}' := F_A^{-1}  e^{\Phi/h}\bar\pl^{-1} e^{-2i\psi/h}\chi |F_A|^2 b  + h (1-\chi)|\bar F_A  e^{\bar\Phi/h} \frac{b}{\bar\pl \bar\Phi}\\\nonumber
\end{eqnarray}
where $\bar\partial^{-1} : C^\infty_0(\supp(\chi)) \to C^\infty(M)$ is the operator constructed in Proposition \ref{surject}. Using Lemma \ref{estimate1} and direct computation gives 
\begin{eqnarray}
\label{LXV of u0'}
\|e^{-\phi/h} u_0'\| \leq C h^{\frac{1}{2}+\epsilon},\ \ \ e^{-\phi/h}L_{X,V}( u_{0}') = O_{L^2}(h^{\frac{1}{2}+ \epsilon}).
\end{eqnarray}
We now compute the boundary value of $u_0'$ along $\Gamma_0$.
\begin{lemma}
\label{bc of ansatz}
  The boundary value for the ansatz $u_0$ in \eqref{def of u0 and u0'} has the boundary condition
\[u_0' \mid_{\Gamma_0} =F_A^{-1}e^{\Phi/h}( hf_0  + h e^{-2i\psi(p_0)/h} f_1 + h^2  f_h) \]
for some $f_0$ and $f_1$ in $C^\infty(\partial M_0)$ independent of $h$ and $f_h$ satisfies $\|f_h\|_{C^k(\partial M_0)} \leq C$.
\end{lemma}
\noindent{\bf Proof.} Along the subset $\Gamma_0 \subset \partial M_0$ we have that $F_A^{-1}\mid_{\Gamma_0} = \overline{F_A}\mid_{\Gamma_0}$ and $\Phi\mid_{\Gamma_0} \in \R$. Therefore, along $\Gamma_0$ the ansatz $u_0'$ has the expression
\[u_0'\mid_{\Gamma_0} = F_A^{-1}  e^{\Phi/h}(\bar\pl^{-1} e^{-2i\psi/h}\chi |F_A|^2 b  + h  \frac{b}{\bar\pl \bar\Phi})\mid_{\Gamma_0}.\]
The boundary value along $\Gamma_0$ of the second term of $u_0'$ can be written down directly. 
 For the first term in \eqref{def of u0 and u0'}, let $\chi'$ be a smooth function on $M_0$ whose support is disjoint from that of $\chi$ and $\chi'=1$ in a neighbourhood of $\partial M_0$. By Proposition \ref{surject} we have that $\chi'\bar\partial^{-1} \chi$ is an operator with smooth kernel. Therefore in a coordinate system which identifies $p_0$ with the origin, $\chi'\bar\partial^{-1} \chi |F_A|^2 b$ has the following expression for some smooth compactly supported $K : \D\times \D \to \C:$
\[(\chi'\bar\partial^{-1} \chi |F_A|^2 b)(\tilde z) =  \int_{\D} e^{-2i\psi(z)/h}\chi(z) K(z,\tilde z) |F_A(z)|^2 dz\wedge d\bar z.\]
We may assume that the support of $\chi$ is chosen to be so small such that we can apply Morse Lemma we obtain a change of variable $ w = \gamma(z)$ with $\gamma(0) = 0$ such that
\begin{eqnarray*}(\chi'\bar\partial^{-1} \chi |F_A|^2 b)(\tilde z)  &=& e^{-2i\psi(p_0)/h} \int_\D e^{2i \langle w, Q w\rangle/h} \chi(w)  \tilde K( w, \tilde z)|F_A(w)|^2 dw\wedge d\bar w
\end{eqnarray*}
for some diagonal matrix $Q$ with entries $\pm 1$ on the diagonal.

With this quadratic phase we can compute explicitly both the principal and the remainder term in the stationary phase expansion. That is,
{\small \begin{eqnarray}
\label{f12}
(\chi'\bar\partial^{-1} \chi |F_A|^2 b)(\tilde z)  
= he^{-2i\psi(p_0)/h} \tilde K(0,\tilde z) +  h^2\int_0^1 (1-t) J(th,  \tilde K(\cdot,\tilde z)) dt
\end{eqnarray}}
where 
\[J(h,  \tilde K(\cdot,\tilde z)) = \int e^{ih\langle \xi, Q^{-1}\xi \rangle}\langle \xi, Q^{-1}\xi \rangle \cal{F}_{w} (\tilde K(w, \tilde z)\chi(w)|F_A(w)|^2)(\xi)d\xi\wedge d\bar\xi\]
with $\cal{F}_{w}$ denoting the classical Fourier transform with respect to the variable $w$. We claim that $J(h,\tilde K(\cdot,\tilde z))$, is a smooth function in $\tilde z$ whose $C^k(M_0)$ norm is bounded independently of $h>0$. Indeed, for any multi-index $\beta$ standard oscillatory integral arguments give
\begin{eqnarray*}D_{\tilde z}^\beta J(h,  \tilde K(\cdot,\tilde z)) &=& \int e^{ih\langle \xi, Q^{-1}\xi \rangle}\langle \xi, Q^{-1}\xi \rangle \cal{F}_{w} (D_{\tilde z}^\beta\tilde K(w, \tilde z)\chi(w)|F_A(w)|^2)(\xi)d\xi\wedge d\bar\xi\\
&=&\int e^{ih\langle \xi, Q^{-1}\xi \rangle} \langle \xi\rangle^{-1}\cal{F}_{w} (a_3(D_w) (D_{\tilde z}^\beta\tilde K(w, \tilde z)\chi(w)|F_A(w)|^2))(\xi)d\xi\wedge d\bar\xi
\end{eqnarray*}
for some constant coefficient third order pseudodifferential operator $a_3$ in the variable $w$. Using the fact that $D_{\tilde z}^\beta\tilde K(w, \tilde z)\chi(w)|F_A(w)|^2$ is a compactly supported $W^{3,p}$ function in $w$ for all $p\in [1,\infty)$ we can estimate the right side by using Holder's inequality
\begin{eqnarray*}|D_{\tilde z}^\beta J(h,  \tilde K(\cdot,\tilde z))| \leq \|\langle \xi\rangle^{-1}\|_{L^p} \|a_3(D_w) D_{\tilde z}^\beta\tilde K(w, \tilde z)\chi(w)|F_A(w)|^2\|_{L^p_w}\end{eqnarray*}
The fact that $K(w,\tilde z)$ is smooth and compactly supported in both variables gives the desired uniform estimate in $\tilde z$.

Plugging this estimate into \eqref{f12} we conclude that 
\[F_A^{-1}\bar\partial^{-1} \chi |F_A|^2 b \mid_{\partial M_0}= F_A^{-1}(he^{-2i\psi(p_0)/h} f_1 + h^2 f_h)\]
where $f_1\in C^\infty(\partial M_0)$ and $\|f_h\|_{C^k(\partial M_0)} \leq C$ for all $h>0$. This completes the proof.\qed\\
Note that since $F_A^{-1} = \bar F_A$ on $\Gamma_0$, we can apply Corollary \ref{surject} and construct holomorphic functions $a_0$, $a_1$, $a_h$, and antiholomorphic functions $\tilde a_0$, $\tilde a_1$, $\tilde a_h$ such that 
\[ F_A^{-1} a_j + \bar F_A \tilde a_j = F_A^{-1} f_j \ \ \text{on}\ \ \Gamma_0, \ \ j=1,2,h.\]
Furthermore, as all the $C^k$ norm of $f_j$ are bounded, we apply the estimates in Corollary \ref{surject} to get that $ \|a_j\|_{C^k(M_0)}\ + \|\tilde a_j \|_{C^k(M_0)} \leq C$ independent of $h>0$. Therefore by \eqref{def of u0 and u0'} and \eqref{LXV of u0'} we have that the ansatz
{\small
\begin{eqnarray}
\label{def of u0''}
u_0'' := u_0' - h\big(e^{\Phi/h}F_A^{-1} (a_0 + e^{-2i\psi(p_0)/h} a_1+ ha_h) + e^{\bar\Phi/h}\bar F_A (\tilde a_0 + e^{-2i\psi(p_0)/h} \tilde a_1+ h \tilde a_h)\big) 
\end{eqnarray}
}
with $u_0'$ given by \eqref{def of u0 and u0'} satisfies 
\begin{eqnarray}
\label{LXV of u0''}
\|e^{-\phi/h} u_0''\| \leq C h^{\frac{1}{2}+\epsilon},\ \ \ e^{-\phi/h}L_{X,V}u_0'' = O_{L^2}(h^{\frac{1}{2} + \epsilon})\ \ {\rm in}\ \ M_0,\ \ \ u_0''\mid_{\Gamma_0} = 0.
\end{eqnarray}

Extend the $O_{L^2}(h^{\frac{1}{2}+\epsilon})$ remainder on the right side trivially to $M_0'$ and applying Corollary \ref{H1 solvability} with we arrive at the following 
\begin{prop}
\label{cgo for interior recover}
There exists solutions to $L_{X,V} u = 0$ in $M_0$ of the form
\[u = u_0'' + e^{\phi/h} r,\ \ u\mid_{\Gamma_0} =0,\ \ \ \|r\|+ \|hdr\| \leq Ch^{1 +\epsilon},\ \ \ \|e^{-\phi/h} u\| \leq C h^{\frac{1}{2}+\epsilon}\]
where $u_0''$ be the ansatz given by \eqref{def of u0''}. 
\end{prop} 
Direct computation gives the following Lemma
\begin{lemma}
\label{dbar of solution}

Let $u$ be the solution to $L_{X,V} u = 0$ constructed in Proposition \ref{cgo for interior recover}. We then have that 
{\Small
\begin{eqnarray}
\bar\partial F_A u = e^{\bar\Phi/h} |F_A|^2 b- h \bar\partial \big(e^{\bar\Phi/h}|F_A|^2 (\tilde a_0 + e^{-2i\psi(p_0)/h} \tilde a_1+ h \tilde a_h) \big)+ h e^{\bar\Phi/h} R_0 + \bar\partial (e^{\phi/h} F_A r)
\end{eqnarray}}
for some $R_0 \in L^\infty(M_0)$ and $r$ satisfying the estimate  $\|r\|+ \|hdr\| \leq Ch^{1 +\epsilon}$. 
\end{lemma}
\end{section}
\begin{section}{Construction of CGO - Part II}
In this section we construct complex geometric optics to recover the zeroth order term of the operator $L_{X_j, V_j}$. The presentation here is essentially a repeat of \cite{duke} and we only include it here for completeness and convenience of the reader. Let $p_0\in {\rm int}(M_0)$ be the critical point of a Morse holomorphic function $\Phi=\phi+i\psi$ on $M$ which is purely real on $\Gamma_0$. By Proposition \ref{criticalpoints} such points form a dense subset of $M$. Given such a holomorphic function, the purpose of this section is to construct, for $X\in W^{3,p}(M_0)$ and $V\in W^{2,p}(M_0)$, solutions $u$ on $M_0$ of $((d+iX)^*(d+iX) +V)u = 0$ of the form
\begin{equation}
\label{cgo}
u = \big( e^{\Phi/h}(F_A^{-1}a + \bar F_Ar_1) + e^{\bar\Phi/h}(\bar F_A\bar a +  F_A^{-1}r_1') + h e^{\Phi/h}F_A^{-1}a_0 + h e^{\bar\Phi/h} \bar F_A \tilde{a}_0\big) + e^{\phi/h} r_2
\end{equation}
with $u\mid_{\Gamma_0} = 0$ for $h>0$ small, where $a$ is holomorphic and $F_A\in W^{4,p}(M_0)$ is a non-vanishing function solving $\bar\partial F_A = iA F_A$, $\bar a_0,\tilde a_0\in H^{2}(M_0)$ are antiholomorphic, moreover $a(p_0)\not=0$ and $a$ vanishes to high order at all other critical points $ p' \in M_0$ of $\Phi$. Furthermore, we ask that the holomorphic function $a$ is purely imaginary on $\Gamma_0$. The existence of such a holomorphic function is a consequence of Lemma \ref{control the zero}. 

The remainder terms $r_1,r_1',r_2$ will be controlled as $h\to 0$ and have particular properties near the critical points of $\Phi$.
More precisely, $r_2$ will be a $O_{L^2}(h^{3/2}|\log h|)$ and $r_1,r_1'$ will be of the form $ h\til{r}_{12} + o_{L^2}(h)$ and $h\til{r}_{12}' + o_{L^2}(h)$ respectively where $\til{r}_{12}, \til{r}_{12}'$ are independent of $h$, 
which can be used to obtain sufficient informations from the stationary phase method in the identification process.

\begin{subsection}{Construction of $r_1$}\label{constr1}
We shall  construct $r_1$ to satisfy
\[e^{-\Phi/h}((d+ iX)^*(d+iX) +V)e^{\Phi/h}(F_A^{-1}a + \bar F_Ar_1) = O_{L^2}(h|\log h|)\]
and $r_1 = r_{11} + h r_{12}$.
Using \eqref{LXV in in dbar form} we can write, for some $Q, \tilde Q \in W^{2,p}(M_0)$
\[L_{X,V} = -2i \star \bar F_A \partial |F_A|^{-2} \bar\partial F_A +  Q = -2i\star F_A^{-1}\bar\partial |F_A|^2 \partial \bar F_A^{-1} + \tilde Q\]
where $A = \pi_{0,1} X$ and $F_A \in W^{4,p}(M_0)$ is a non-vanishing function solving $\bar\partial F_A = iAF_A$ and unitary along $\Gamma_0$. Such functions are given by Proposition \ref{new boundary integral id}.

We let $G$ be the Green operator of the Laplacian on the smooth surface with boundary 
$M_0$ with Dirichlet condition, so that $\Delta_gG={\rm Id}$ on $L^2(M_0)$. In particular this implies 
that $\bar{\partial}\partial G=\frac{i}{2}\star^{-1}$ where $\star^{-1}$ is the inverse of $\star$ mapping functions to $2$-forms.
We will search for $r_{1}\in H^{2}(M_0)$ satisfying
$||r_1||_{L^2}=O(h)$ and
\begin{equation}
\label{dequation}
e^{-2i\psi/h}|F_A|^2\partial e^{2i\psi/h} r_1 = -\pl G (aQ) + \omega + O_{H^1}(h|\log h|)
\end{equation}
where $\omega$ is a smooth holomorphic 1-form on $M_0$. 
Indeed, using the fact that $\Phi$ is holomorphic we have
\[e^{-\Phi/h}L_{X,V}e^{\Phi /h}=-2i\star F_A^{-1}\bar\partial e^{-\Phi/h} |F_A|^2 \partial \bar F_A^{-1}e^{\Phi/h} + Q=-2i\star F_A^{-1}\bar\partial e^{-2i\psi/h} |F_A|^2 \partial \bar F_A^{-1}e^{2i\psi/h} + \tilde Q\]
for some $Q, \tilde Q \in W^{2,p}(M_0)$. Applying $-2i\star\bar{\pl}$ to \eqref{dequation}, we obtain (note that $\pl G(aQ)\in C^{2,\alpha}(M_0)$ by elliptic regularity)
\[e^{-\Phi/h}L_{X,V}e^{\Phi/h}\bar F_Ar_1=-aQ+O_{L^2}(h|\log h|).\]
We will choose $\omega$ to be a smooth holomorphic $1$-form on $M_0$ such that at all 
critical point $p'$ of $\Phi$ in $M_0$, the form $\beta:=\pl G(aQ) - \omega$  with value in $T^*_{1,0}M_0$ vanish to the highest possible order.
Writing $\beta=\beta(z)dz$  in local complex coordinates, $\beta(z)$ is $C^{2+\alpha}$ by elliptic regularity and 
we have $-2i\pl_{\bar{z}}\beta(z)=aV$, therefore $\pl_z\pl_{\bar{z}}\beta(p')=\pl^2_{\bar{z}}\beta(p')=0$ at each critical point $p'\not=p_0$ by construction 
of the function $a$.  Therefore, we deduce that at each critical point $p'\neq p_0$, $\pl G(aQ)$ has Taylor series expansion $\sum_{j = 0}^2 c_j z^j + O(|z|^{2+\alpha})$. That is, all the lower order terms of the Taylor expansion of $\pl G(aQ)$ around $p'$ are polynomials of $z$ only. 
\begin{lemma}\label{formomega}
Let $\{p_0,...,p_N\}$ be finitely many points on $M_0$ and let $\theta$ be a $C^{2,\alpha}$ section of $T^*_{1,0}M_0$.
Then there exists a $C^k$ holomorphic function $f$ on $M_0$ with $k\in\nn$ large, such that $f$ vanishes to high order at the points $\{p_1,...,p_N\}$ and $\omega = \partial f$ satisfies the following: in complex local coordinates $z$ near 
$p_0$ , one has $\pl_z^\ell\theta(p_0)=\pl_z^\ell\omega(p_0)$ for $\ell=0,1,2$, where $\theta=\theta(z)dz$ and $\omega=\omega(z)dz$.
\end{lemma}
\noindent{\bf Proof}. This is a direct consequence of Lemma \ref{control the zero}.
\qed\\
Applying this to the form $\pl G(aQ)$ and using the observation we made above, we can construct a $C^k$ holomorphic form $\omega$ such that in local coordinates $z$ centered at a critical point 
$p'$ of $\Phi$ (i.e $p'=\{z=0\}$ in this coordinate), we have for $\beta=-\pl G(aQ)+\omega=\beta(z)dz$
\begin{equation}\label{decayofb}
\begin{gathered}
|\pl_{\bar{z}}^m\pl^{\ell}_z \beta(z)|=O(|z|^{2+\alpha-\ell-m}) , \textrm{ for } \ell+m\leq 2 ,\quad \textrm{ if }p'\not=p_0\\
|\beta(z)|=O(|z|) , \quad \textrm{ if }p'=p_0.  
 \end{gathered}
 \end{equation}

Now, we let $\chi_1\in C_0^\infty(M_0)$ be a cutoff function supported in a small neighbourhood $U_{p_0}$ of the critical point $p_0$ and identically $1$ near $p_0$, and 
$\chi\in C_0^\infty(M_0)$ is defined similarly with $\chi =1$ on the support of $\chi_1$.
We will construct $r_1= r_{11} +h r_{12}$ in two steps : first, we will construct $r_{11}$ to solve equation \eqref{dequation} 
locally near the critical point $p_0$ of $\Phi$ and then we will 
construct the global correction term $r_{12}$ away from $p_0$ by using the extra vanishing of $\beta$ in \eqref{decayofb} at the other critical points.

We define locally in complex coordinates centered at $p_0$ and containing the support of $\chi$
\[r_{11}:=\chi e^{-2i\psi/h}R(e^{2i\psi/h}\chi_1|F_A|^{-2}\beta)\] 
where $Rf(z) := -(2\pi i)^{-1}\int_{\R^2} \frac{1}{\bar{z}-\bar{\xi}}f d\bar{\xi}\wedge d\xi$ for $f\in L^\infty$ compactly supported 
is the classical Cauchy-Riemann operator inverting locally $\pl_z$ ($r_{11}$ is extended by $0$ outside the neighbourhood of $p$).
The function $r_{11}$ is in $C^{3+\alpha}(M_0)$ and we have 
\begin{equation}\begin{gathered}\label{r11}
e^{-2i\psi/h}\pl(e^{2i\psi/h}r_{11}) = \chi_1(-\pl G(aQ) + \omega) + \eta\\
\textrm{ with }\eta:= e^{-2i\psi/h}R(e^{2i\psi/h}\chi_1\beta|F_A|^{-2})\pl\chi.
\end{gathered}
\end{equation}
We then construct $r_{12}$ by observing that  $\beta$ vanishes to order $2+\alpha$ at critical points of $\Phi$ other than $p$ (from \eqref{decayofb}), and 
$\pl \chi=0$ in a neighbourhood of any critical point of $\psi$, so we can find $r_{12}$ satisfying
\[2ir_{12}\pl\psi = (1-\chi_1)\beta |F_A|^{-2}.\] 
This is possible since both $\pl\psi$ and the right hand side are valued in $T^*_{1,0}M_0$, $\pl \psi$ has finitely many isolated zeroes on $M_0$: 
$r_{12}$ is then a function which is in $C^{2,\alpha}(M_0\setminus{P})$ where $P:=\{p_1,\dots, p_N\}$ is the set of critical points other than $p_0$,
it extends to a $C^{1,\alpha}(M_0)$  and it satisfies in local complex coordinates $z$ near each $p_j$ 
\[ |\pl_{\bar{z}}^m\pl_z^l r_{12}(z)|\leq C|z-p_j|^{1+\alpha-m-l} , \quad m+l\leq 2.\]
by using also the fact that $\pl \psi$ can be locally be considered as holomorphic function with a zero of order $1$ at each $p_j$.
This implies that $r_1\in H^2(M_0)$ and  we have 
\[ e^{-2i\psi/h}|F_A|^2\pl(e^{2i\psi/h}r_1) = \beta+h\pl r_{12}+\eta=-\pl G(aQ)-\omega+ h\pl r_{12} + \eta.\]
Now the first error term $||\pl r_{12}||_{H^1(M_0)}$ is bounded by 
\[||\pl r_{12}||_{H^1(M_0)}\leq C\left( \left|\left|\frac{(1-\chi_1)b(z)}{\pl_z \psi(z)}\right|\right|_{H^2(U_{p_0})}\right)\leq
C \]
for some constant $C$, where we used the fact that $\frac{(1-\chi_1)\beta(z)}{\pl_z \psi(z)}$ is in $H^2(U_{p_0})$ and independent of $h$.
To deal with the $\eta$ term, we need the following 
\begin{lemma}\label{termeta}
The following estimates hold true 
\[\begin{gathered} 
||\eta||_{H^2}=O(|\log h|),\quad \|\eta\|_{H^1}\leq O(h|\log h|),  \quad ||r_{1}||_{L^2}=O(h), \quad ||r_1-h\til{r}_{12}||_{L^2}=o(h) 
\end{gathered} \]
where $\til{r}_{12}$ solves $2i\til{r}_{12}\pl\psi = \beta|F_A|^{-2}$ is independent of $h$ and $H^2$ near the boundary $\partial M_0$.
\end{lemma}
\textbf{Proof}. We start by observing that since $\beta$ vanishes to high order at all critical points of $\Phi$ except for the interior point $p_0\in M$, one has that $\til{r}_{12}$ is in $H^2$ in a neighbhourhood of the boundary $\partial M$. Furthermore, 
\begin{equation}\label{decomposition} 
\begin{gathered}
||r_1-h\til{r}_{12}||_{L^2}= \left| \left| \chi e^{-2i\psi/h}R (e^{2i\psi/h}\chi_1 \beta|F_A|^{-2})-h\frac{\chi_1\beta|F_A|^{-2}}{2i\pl_z\psi}\right|\right|_{L^2(U_p)},\\
||r_1||_{L^2}\leq \left| \left| \chi e^{-2i\psi/h}R (e^{2i\psi/h}\chi_1 \beta|F_A|^{-2})-h\frac{\chi_1\beta|F_A|^{-2}}{2i\pl_z\psi}\right|\right|_{L^2(U_p)}+
h||\til{r}_{12}||_{L^2(M_0)}
\end{gathered}
\end{equation}
The first term is estimated in Proposition 2.7 of \cite{IUY}, it is a $o(h)$, while the $||\til{r}_{12}||_{L^2}$ is 
independent of $h$. 
Now are going to estimate the $H^2$ norms of $\eta$. Locally in complex coordinates $z$ centered at $p_0$ (ie. $p_0=\{z=0\}$), we have 
\begin{equation} \label{definitioneta}
\eta(z)= -\pl_z\chi(z)e^{-\frac{2i\psi(z)}{h}}\int_{\cc} e^{\frac{2i\psi(\xi)}{h}}\frac{1}{\bar{z}-\bar{\xi}} \chi_1(\xi)\beta(\xi) |F_A(\xi)|^{-2}\frac{d\xi_1d\xi_2}{\pi} , 
\quad \xi=\xi_1+i\xi_2.
\end{equation} 
Since $\beta$ is $C^{2,\alpha}$ in $U$, we decompose $\beta(\xi)=\cjg\nabla \beta(0),\xi\cjd+\til{\beta}(\xi)$ using Taylor formula, so we have 
$\til{\beta}(0)=\pl_\xi \til{\beta}(0)=0$ and we split the integral \eqref{definitioneta} with $\cjg\nabla \beta(0),\xi\cjd$ and $\til{\beta}(\xi)$.  
Since the integrand with the $\cjg\nabla \beta(0),\xi\cjd$ is smooth and compactly supported in $\xi$ (recall that $\chi_1=0$ on the support of $\pl_z\chi$),
we can apply stationary phase to get that 
\[\left|\pl_z\chi(z)e^{-\frac{2i\psi(z)}{h}}\int_{\cc} e^{\frac{2i\psi(\xi)}{h}}\frac{1}{\bar{z}-\bar{\xi}} \chi_1|F_A(\xi)|^{-2}(\xi)\cjg\nabla b(0),\xi\cjd \frac{d\xi_1d\xi_2}{\pi}\right|\leq Ch^2\]
uniformly in $z$.
Now set $\til{\beta}_z(\xi)=\pl_z\chi(z)\chi_1(\xi)\til{\beta}(\xi)/(\bbar{z-\xi})$ which is $C^{2,\alpha}$ in $\xi$ and smooth in $z$. 
Let $\theta\in C_0^\infty([0,1))$ be a cutoff function which is equal to $1$ near $0$ and 
set $\theta_h(\xi):=\theta(|\xi|/h)$, then we have by integrating by parts 
\begin{equation}\label{intbyparts}
\begin{split}
\int_{\cc} e^{\frac{2i\psi(\xi)}{h}}|F_A(\xi)|^{-2}\til{\beta}_z(\xi) d\xi_1d\xi_2= &
h^2\int_{{\rm supp}(\chi_1)} e^{\frac{2i\psi(\xi)}{h}}\pl_{\bar{\xi}}\left(\frac{1-\theta_h(\xi)}{2i\pl_{\bar{\xi}} \psi}
\pl_{\xi}\left(\frac{|F_A(\xi)|^{-2}\til{\beta}_z(\xi)}{2i\pl_\xi\psi}\right)\right) d\xi_1d\xi_2\\
&-h\int_{{\rm supp}(\chi_1)} e^{\frac{2i\psi(\xi)}{h}}\theta_h(\xi)\pl_{\xi}\left(\frac{|F_A(\xi)|^{-2}\til{\beta}_z(\xi)}{2i\pl_\xi\psi}\right) d\xi_1d\xi_2 .
\end{split}\end{equation}
Using polar coordinates with the fact that $\til{\beta}_z(0)=0$, 
it is easy to check that the second term in \eqref{intbyparts} is bounded uniformy in $z$ 
by $Ch^{2}$. 
To deal with the first term, we use $\til{\beta}_z(0)=\pl_\xi \til{\beta}_z(0)=\pl_{\bar{\xi}}\til{\beta}_z(0)=0$ and 
a straightforward computation in polar coordinates shows
that the first term of \eqref{intbyparts}  is bounded uniformly in $z$ by $Ch^{2}|\log(h)|$.
We conclude that
\[ ||\eta||_{L^2}\leq C||\eta||_{L^\infty}\leq Ch^{2}|\log h|.\]
It is also direct to see that the same estimates holds with a loss of $h^{-2}$ for any derivatives in $z,\bar{z}$ of order less or equal to $2$, 
since they only hit the $\chi(z)$ factor, the $(\bar{z}-\bar{\xi})^{-1}$ factor or the oscillating term $e^{-2i\psi(z)/h}$.
So  we deduce that 
\[ ||\eta||_{H^2}= O(|\log h|).\]
and this ends the proof.
\qed\\

We summarize the result of this section with the following
\begin{lemma}
\label{solve equation to next order}
Let $k\in\nn$ be large and $\Phi\in C^k(M_0)$ be a holomorphic function on $M_0$ which is Morse in $M_0$ with a critical point at $p_0\in {\rm int}(M_0)$. 
Let $a\in C^k(M_0)$ be a holomorphic function on $M_0$ purely imaginary on $\Gamma_0$ and vanishing to high order at every critical point of $\Phi$ other than $p$. 
Then there exists $r_1 \in H^{2}(M_0)$ such that $r_1 = h \tilde{r}_{12} + o_{L^2}(h)$ with $\tilde{r}_{12}\in L^2$ independent of $h$ and
\[e^{-\Phi/h}L_{X,V}e^{\Phi/h}(F_A^{-1}a + \bar F_Ar_1) =O_{L^2}(h|\log h|).\]
\end{lemma}
One can follow the same construction for the antiholomorphic phase $\bar\Phi$ in place of $\Phi$. Indeed, repeating the above argument in this case yields
\begin{lemma}
\label{antiholo solve equation to next order}
Let $k\in\nn$ be large and $\Phi\in C^k(M_0)$ be a holomorphic function on $M_0$ which is Morse in $M_0$ with a critical point at $p_0\in {\rm int}(M_0)$. 
Let $a\in C^k(M_0)$ be a holomorphic function on $M$ purely imaginary on $\Gamma_0$ and vanishing to high order at every critical point of $\Phi$ other than $p$. 
Then there exists $r_1' \in H^{2}(M_0)$ such that $r'_1 = h \tilde{r}_{12}' + o_{L^2}(h)$ with $\tilde{r}_{12}'\in L^2$ independent of $h$ and
\[e^{-\bar\Phi/h}L_{X,V}e^{\bar\Phi/h}(\bar F_A\bar a +  F_A^{-1}r'_1) =O_{L^2}(h|\log h|).\]
\end{lemma}

\end{subsection}
\begin{subsection}{Construction of $a_0$}
We have constructed the correction terms $r_1$ which solves the Schr\"odinger equation to order $h$ as stated in Lemma \ref{solve equation to next order}. In this subsection, we will construct  a holomorphic function $a_0$ which annihilates the boundary value of the solution on 
$\Gamma_0$. In particular, we have the following
\begin{lemma}
\label{a0 and a1}
There exists a holomorphic function $a_0\in H^2(M_0)$ and an antiholomorphic function $\tilde a_0\in H^2(M_0)$  independent of $h$ such that 
\[e^{-\Phi/h}L_{X,V}\big( e^{\Phi/h}(F_A^{-1}a + \bar F_Ar_1) + e^{\bar\Phi/h}(\bar F_A\bar a +  F_A^{-1}r_1') + h e^{\Phi/h}F_A^{-1}a_0 + h e^{\bar\Phi/h} \bar F_A \tilde{a}_0\big)= O_{L^2}(h|\log h|)\]
and 
\[\big( e^{\Phi/h}(F_A^{-1}a + \bar F_Ar_1) + e^{\bar\Phi/h}(\bar F_A\bar a +  F_A^{-1}r_1') + h e^{\Phi/h}F_A^{-1}a_0 + h e^{\bar\Phi/h} \bar F_A \tilde{a}_0\big)|_{\Gamma_0} =0.\]
\end{lemma}
\noindent{\bf Proof}. First, notice that $h^{-1}r_1|_{\pl M_0}=\til{r}_{12}|_{\pl M_0}\in H^{3/2}(\pl M_0)$ and $h^{-1}r_1'|_{\pl M_0}=\til{r}_{12}'|_{\pl M_0}\in H^{3/2}(\pl M_0)$ are independent of $h$. Using part (iii) of Proposition \ref{surject} one can construct $a_0, \tilde a_0 \in H^2(M_0)$ holomorphic and antiholomoprhic respectively such that $[a_0 + \tilde a_0]\mid_{\Gamma_0} =- (\til{r}_{12} + \til{r}_{12}')\mid_{\Gamma_0}$. 
Since $\Phi$ is purely real on $\Gamma_0$ and $F_A$ is unitary on $\Gamma_0$, we see that 
\[\big( e^{\Phi/h}(F_A^{-1}a + \bar F_Ar_1) + e^{\bar\Phi/h}(\bar F_A\bar a +  F_A^{-1}r_1') + h e^{\Phi/h}F_A^{-1}a_0 + h e^{\bar\Phi/h} \bar F_A \tilde{a}_0\big)|_{\Gamma_0} =0.\]
This combined with the asymptotic given by Lemma \ref{solve equation to next order} and Lemma \ref{antiholo solve equation to next order} completes the proof.\qed\\
\end{subsection}
We can extend the $O_{L^2}(|h \log h|)$ remainder in Lemma \ref{a0 and a1} trivially to all of $M_0'$ and apply Corollary \ref{H1 solvability} to obtain the following CGO:
\begin{proposition}
\label{completecgo}
There exist solutions to $L_{X,V}u = 0$ with boundary condition $u|_{\Gamma_0} = 0$ of the form \eqref{cgo} with $r_1$, $r_1'$, $a_0$, $\tilde{a}_0$ constructed in the previous sections and $r_2$ satisfying $\|r_2\|_{H^1_{scl}}= O(h^{3/2}|\log h|)$.
\end{proposition}

\end{section}

\begin{section}{Recovery of Coefficients}
\begin{subsection}{Recovering the Modulus of $F_{A_j}$} We assume that ${\cal C}_{X_1, V_1, \partial M_0\backslash \Gamma} = {\cal C}_{X_2, V_2, \partial M_0\backslash \Gamma}$. By Proposition \ref{new boundary integral id} we have that there exists a portion of the boundary $\Gamma_0$ containing $\Gamma$ whose complement contains an open set and non-vanishing solutions $F_{A_j} \in W^{2,p}(M_0)\cap W^{3,p}_{loc}(M_0)$ to $\bar\partial F_{A_j} = A_j F_{A_j}$ with $|F_{A_j}|\mid_{\Gamma_0} = 1$ such that $F_{A_1}\mid_{\partial M_0\backslash \Gamma_0} = F_{A_2}\mid_{\partial M_0\backslash \Gamma_0}$. Furthermore, if $L_{X_j,V_j}u_j =0$ with $u_j \mid_{\Gamma_0} = 0$ then the boundary integral identity
\[0=\int_{M_0} \langle (|F_{A_1}|^{-2} - |F_{A_2}|^{-2})\bar\pl \tilde u_1, \bar\pl \tilde u_2\rangle +  \frac{1}{2}\langle (Q_2 |F_{A_2}|^{2} -  Q_1 |F_{A_1}|^{2})\tilde u_1,\tilde u_2\rangle \]
holds for $\tilde u_j := F_{A_j} u_j$.

The main result of this subsection is to show that the $F_{A_1}$ and $F_{A_2}$ chosen above have the same modulus. More precisely,
\begin{proposition}
\label{|FA1| = |FA2|}
If ${\cal C}_{X_1, V_1, \partial M_0\backslash \Gamma} = {\cal C}_{X_2, V_2, \partial M_0\backslash \Gamma}$ and $F_{A_j}$ are chosen as above then $|F_{A_1}| = |F_{A_2}|$.
\end{proposition}
\noindent{\bf Proof. }If $\hat p$ is any interior point of $M_0$ and $B_\eps(\hat p)$ is a neighbourhood of the point, then by Proposition \ref{criticalpoints} there exists a Morse holomorphic function $\Phi = \phi + i\psi$ on $M$ which is real valued along $\Gamma_0$ with a critical point $p_0$ in $B_\eps(\hat p)$. If $\{p_0,.., p_N\}$ are the critical points of $\Phi$, we can construct by Lemma \ref{control the zero} an antiholomorphic 1-form $b$ which vanishes to order $k$ at $\{p_1,.., p_N\}$ and $b(p_0)\neq 0$. We have the following Lemma which we will prove at the end of the subsection:
\begin{lemma}
\label{main term plus o(h)}
For all such $\Phi$ and $b$ we have the following asymptotic as $h\to 0$:
\begin{eqnarray}
\label{main term plus o(h) identity}
0 = \int_{M_0} (|F_{A_2}|^2 - |F_{A_1}|^2) |b|^2 e^{-2i\psi/h} + o(h)
\end{eqnarray}
\end{lemma}
Since $b$ vanishes at all critical points of $\Phi$ except for $p_0$, \eqref{main term plus o(h) identity} has stationary phase expansion
\[0 = he^{2i\psi(p_0)/h} (|F_{A_2}(p_0)|^2 - |F_{A_1}(p_0)|^2) + o(h)\]
which implies that $|F_{A_2}(p_0)|^2 - |F_{A_1}(p_0)|^2=0$. Since $\eps>0$ can be chosen arbitrarily small, the continuity of $F_{A_j}$ then gives that $|F_{A_2}(\hat p)|^2 =|F_{A_1}(\hat p)|^2$ for any $\hat p \in M_0$.\qed

It remains to prove Lemma \ref{main term plus o(h)}.\\
\noindent{\bf Proof of Lemma \ref{main term plus o(h)}.} By Proposition \ref{new boundary integral id} we have that if $L_{X_j,V_j}u_j = 0$ and $u_j\mid_{\Gamma_0} = 0$ then
\[0=\int_{M_0} \langle (|F_{A_1}|^{-2} - |F_{A_2}|^{-2})\bar\pl \tilde u_1, \bar\pl \tilde u_2\rangle +  \frac{1}{2}\langle (Q_2 |F_{A_2}|^{2} -  Q_1 |F_{A_1}|^{2})\tilde u_1,\tilde u_2\rangle \]
where $\tilde u_j = F_{A_j}u_j$ and $Q_j = *dX_j + V_j$.

If $\Phi$ and $b$ are as given in the statement of the Lemma, let $u_1$ be the solution to $L_{X_1,V_1} u_1 = 0$ given by Proposition \ref{cgo for interior recover} for the phase $\Phi$ and let $u_2$ be the solution to $L_{X_2,V_2} u_2 = 0$ given by Proposition \ref{cgo for interior recover} for the phase $-\Phi$. That is,
\[u_1 = u_{0,+}'' + e^{\phi/h} r_1,\ \ \ u_2 = u_{0,-}'' + e^{-\phi/h}r_2\]
where $u_{0,\pm}''$ are the ansatz given by \eqref{def of u0''} for $\pm\Phi$ respectively. Plugging these solutions into this identity and using the estimate on $u_j$ in Proposition \ref{cgo for interior recover} in conjunction with the identity in Lemma \ref{dbar of solution}, the boundary integral identity becomes
{\Small
\begin{eqnarray}
\label{integral id terms}
0&=&\int_{M_0}  (|F_{A_2}|^2 - |F_{A_1}|^2)|b|^2 e^{-2i\psi/h}\\\nonumber&-&h\int_{M_0} \langle (1 - |F_{A_2}|^{-2}|F_{A_1}|^2)e^{\bar\Phi/h} b, \bar\partial \big(e^{-\bar\Phi/h}|F_{A_2}|^2  {\mathcal A}_h+e^{-\phi/h} F_{A_2} r\big)\rangle\\\nonumber&-&h\int_{M_0} \langle (|F_{A_1}|^{-2}|F_{A_2}|^2 - 1) \bar\partial \big(e^{\bar\Phi/h}|F_{A_1}|^2 {\mathcal A}_h' +e^{\phi/h} F_{A_1} r'\big),e^{-\bar\Phi/h} b \rangle\\\nonumber&+&h^2\int_{M_0} \langle (|F_{A_1}|^{-2} - |F_{A_2}|^{-2}) \bar\partial \big(e^{\bar\Phi/h}|F_{A_1}|^2 {\mathcal A}_h'+e^{\phi/h} F_{A_1} r'\big), \bar\partial \big(e^{-\bar\Phi/h}|F_{A_2}|^2 {\mathcal A}_h+e^{-\phi/h} F_{A_2} r\big) \rangle\\\nonumber&+&o(h) 
\end{eqnarray}}
where ${\mathcal A}_h := \tilde a_0 + e^{-2i\psi(p_0)/h} \tilde a_1+ h \tilde a_h$ and ${\mathcal A}_h' := \tilde a_0' + e^{-2i\psi(p_0)/h} \tilde a_1'+ h \tilde a_h'$ are antiholomorphic functions depending on the parameter $h>0$.

The second term can be estimated by taking the adjoint of $\bar\partial$ and using that \[|F_{A_1}| \mid_{\partial M_0}= |F_{A_2}|\mid_{\partial M_0}\] to obtain
{\small\begin{eqnarray}
\label{second term}
&&h\int_{M_0} \langle (1 - |F_{A_2}|^{-2}|F_{A_1}|^2)e^{\bar\Phi/h} b, \bar\partial \big(e^{-\bar\Phi/h}|F_{A_2}|^2 {\mathcal A}_h+e^{-\phi/h} F_{A_2} r\big)\rangle\\\nonumber &=&- h\int_{M_0} e^{-2i\psi/h}\partial (|F_{A_2}|^{-2}|F_{A_1}|^2)\wedge b\big(|F_{A_2}|^2 {\mathcal A}_h+e^{-i\psi/h} F_{A_2} r\big).\end{eqnarray}}
By Proposition \ref{cgo for interior recover} the remainder $r$ satisfies the estimate $\|r\|\leq Ch^{1+\eps}$. This combined with the fact that $\int e^{2i\psi/h} f = o(1)$ for all $f\in L^1$ independent of $h$ gives that \eqref{second term} can be estimated by
{\small\begin{eqnarray}
\label{second term is o(h)}
&&h\int_{M_0} \langle (1 - |F_{A_2}|^{-2}|F_{A_1}|^2)e^{\bar\Phi/h} b, \bar\partial \big(e^{-\bar\Phi/h}|F_{A_2}|^2  {\mathcal A}_h+e^{-\phi/h} F_{A_2} r\big)\rangle = o(h).
\end{eqnarray}}
We have then that the second term of \eqref{integral id terms} can be estimated by $o(h)$. The third term of \eqref{integral id terms} can be treated the same way to obtain
{\small\begin{eqnarray}
\label{third term is o(h)}
&&h\int_{M_0} \langle (|F_{A_1}|^{-2}|F_{A_2}|^2 - 1) \bar\partial \big(e^{\bar\Phi/h}|F_{A_1}|^2  {\mathcal A}_h' +e^{\phi/h} F_{A_1} r'\big),e^{-\bar\Phi/h} b \rangle= o(h).
\end{eqnarray}}
Therefore, plugging the estimates of \eqref{second term is o(h)} and \eqref{third term is o(h)} into \eqref{integral id terms} we have
{\Small
\begin{eqnarray}
0&=&\int_{M_0}  (|F_{A_2}|^2 - |F_{A_1}|^2)|b|^2 e^{-2i\psi/h}\\\nonumber&+&h^2\int_{M_0} \langle (|F_{A_1}|^{-2} - |F_{A_2}|^{-2}) \bar\partial \big(e^{\bar\Phi/h}|F_{A_1}|^2 {\mathcal A}_h'+e^{\phi/h} F_{A_1} r'\big), \bar\partial \big(e^{-\bar\Phi/h}|F_{A_2}|^2 {\mathcal A}_h+e^{-\phi/h} F_{A_2} r\big) \rangle
\\\nonumber &+&o(h).\end{eqnarray}}

For the remaining integral we integrate by parts again to obtain 
{\small
\begin{eqnarray*}
\nonumber0&=&\int_{M_0}  (|F_{A_2}|^2 - |F_{A_1}|^2)|b|^2 e^{-2i\psi/h}\\\nonumber&-&h^2\int_{M_0}\big(e^{\bar\Phi/h}|F_{A_1}|^2 {\mathcal A}_h'+e^{\phi/h} F_{A_1} r'\big) \langle \bar\partial(|F_{A_1}|^{-2} - |F_{A_2}|^{-2}) , \bar\partial 
\big(e^{-\bar\Phi/h}|F_{A_2}|^2 {\mathcal A}_h+e^{-\phi/h} F_{A_2} r\big) \rangle\\\nonumber &+& h^2 \int_{M_0} \big(e^{\bar\Phi/h}|F_{A_1}|^2 {\mathcal A}_h'+e^{\phi/h} F_{A_1} r'\big)(|F_{A_1}|^{-2} - |F_{A_2}|^{-2}) \Delta_g \big(e^{-\bar\Phi/h}|F_{A_2}|^2 {\mathcal A}_h+e^{-\phi/h} F_{A_2} r\big)
\\ &+&o(h).\end{eqnarray*}}
Using the fact that ${\mathcal A}_h = \tilde a_0 + e^{-2i\psi(p_0)/h} \tilde a_1+ h \tilde a_h$ with $\|a_h\|_{C^k}$ independent of $h$ and \[e^{\phi/h}\Delta_g e^{-\phi/h}r = 2e^{\phi/h}\langle X_2,  d e^{-\phi/h} r\rangle + (V_2 + |X_2|^2 )r + O_{L^2}(h^{\demi+\eps}),\ \  \|r\|_{H^1_{scl}}\leq C h^{1+\eps},\] 
we have that the above expression becomes
\begin{eqnarray*}
0 = \int_{M_0} (|F_{A_2}|^2 - |F_{A_1}|^2) |b|^2 e^{-2i\psi/h} + o(h)
\end{eqnarray*}
and the proof is complete.\qed\\
\end{subsection}
\begin{subsection}{Gauge Equivalence of $X_1$ and $X_2$.} The purpose of this subsection is to prove the first assertion of Theorem \ref{main thm, larger Gamma}. More precisely,
\begin{proposition}
\label{gauge equivalence}There exists an open subset of the boundary $\Gamma_0\subset\partial M_0$ compactly containing $\Gamma$ with $\partial M_0\backslash \bar\Gamma_0$ an open segment and a non-vanishing function $\Theta$ such that 
\[X_1-X_2 =d\Theta/\Theta,\ \ \ \Theta\mid_{\partial M_0\backslash\Gamma_0} = 1.\]
\end{proposition}
\noindent{\bf Proof.} By Lemma \ref{|FA1| = |FA2|} we can choose non-vanishing functions $F_{A_j}$ satisfying $\bar\partial F_{A_j} = iA_j F_{A_j}$ with boundary condition $|F_{A_j}|\mid_{\Gamma_0} = 1$ such that 
\[F_{A_1}\mid_{\partial M_0\backslash \Gamma_0} = F_{A_2}\mid_{\partial M_0\backslash \Gamma_0}\ \text{and}\ |F_{A_1}| = |F_{A_2}|\ \text{in}\ M_0.\]
Observe that if we define ${F}_{\bar{A}_j}:={\bar F_{A_j}}^{-1}$, it is a solution to $\partial {{F}_{\bar{A}_j}} = i\bar A_j {{F}_{\bar{A}_j}}$ with boundary condition $|F_{\bar{A}_j}|\mid_{\Gamma_0} = 1$ such that 
\[ F_{\bar{A}_1}\mid_{\partial M_0\backslash \Gamma_0} =  F_{\bar{A}_2}\mid_{\partial M_0\backslash \Gamma_0}\ \text{and}\ |F_{\bar{A}_1}| = |F_{\bar{A}_2}|\ \text{in}\ M_0.\]
Therefore, $\Theta:=F_{A_1}/F_{A_2}={F}_{\bar{A}_1}/{F}_{\bar{A}_2}$ is a function mapping $M_0$ to the unit circle $S^1\subset \cc$ solving the differential equation
\[\bar{\pl}\Theta/\Theta= i(A_1-A_2), \quad \pl\Theta/\Theta=i(\bar{A}_1-\bar{A}_2)\]
and thus $d\Theta/\Theta=i(X_1-X_2)$ with $\Theta \mid_{\partial M_0\backslash\Gamma_0} = 1$ and the proof is complete.\qed\\
\end{subsection}
\begin{subsection}{Identifying Zeroth Order Term} The purpose of this section is to prove that under the assumptions of Theorem \ref{main thm, larger Gamma}, $V_1= V_2$. In conjunction with Proposition \ref{gauge equivalence} this completes the proof of Theorem \ref{main thm, larger Gamma}. The argument presented here is almost identical to that of of \cite{duke} which we repeat here for the convenience of the reader.

We begin by observing that due to Proposition \ref{gauge equivalence} the operators $d + iX_1$ and $d+iX_2$ are gauge equivalent. Therefore we can assume, by taking a gauge transformation, that $X := X_1 = X_2\in W^{3,p}(M_0)$ and that 
\[{\cal C}_{X, V_1, \partial M_0\backslash \Gamma_0} = {\cal C}_{X,V_2, \partial M_0 \backslash \Gamma_0}.\]
So by repeating the same boundary determination argument in the appendix of \cite{duke} we can conclude that $V_{1}\mid_{\partial M_0\backslash \Gamma_0} =V_{2}\mid_{\partial M_0\backslash \Gamma_0}$.

If we let $\alpha\in  W^{4,p}(M_0)$ be a solution of 
\[\bar\partial \alpha = A := \pi_{0,1}X,\ \ \ i\alpha\mid_{\Gamma_0} \in \R\]
given by Proposition \ref{surject}, and set $F_A = e^{i\alpha}$ we have by Proposition \ref{new boundary integral id}
\[0=  \int_{M_0} (V_2 - V_1 )|F_A|^4 u_1 \bar u_2 \]
for all $u_j$ solving 
\[L_{X,V_j} u_j = 0\ \ \ u_j\mid_{\Gamma_0} = 0\ \ \text{for}\ \  j=1,2.\]
Let $p_0\in M_0$ be an interior point such that there exits a holomorphic Morse function $\Phi$ on $M$ with $\Phi\mid_{\Gamma_0} \in \R$. We also require that $\im(\Phi(p_0))\neq 0$. Such points are dense on $M_0$ by Proposition \ref{criticalpoints}. Let $a$ be a holomorphic function which is purely imaginary on $\Gamma_0$ such that $a(p_0)\neq0$ and $a$ vanishes to high order at all other critical points of $\Phi$. One can construct such a holomorphic function by Lemma \ref{control the zero}. Applying Proposition \ref{completecgo} to both $\Phi$ and $-\Phi$ yields solutions to $L_{X,V_j} u_j = 0$ which are of the form:
\begin{equation*}
u_1 = \big( e^{\Phi/h}(F_A^{-1}a + \bar F_Ar_1) + e^{\bar\Phi/h}(\bar F_A\bar a +  F_A^{-1}r_1') + h e^{\Phi/h}F_A^{-1}a_0 + h e^{\bar\Phi/h} \bar F_A \tilde{a}_0\big) + e^{\phi/h} r_2
\end{equation*}
\begin{equation*}
u_2 = \big( e^{-\Phi/h}(F_A^{-1}a + \bar F_As_1) + e^{-\bar\Phi/h}(\bar F_A\bar a +  F_A^{-1}s_1') + h e^{-\Phi/h}F_A^{-1}a'_0 + h e^{-\bar\Phi/h} \bar F_A \tilde{a}'_0\big) + e^{-\phi/h} s_2
\end{equation*}
where 
\[r_1 = h\til{r}_{12} + o_{L^2}(h), r_1' = h\til{r}'_{12} + o_{L^2}(h), s_1 = h\til{s}_{12} + o_{L^2}(h), s_1' = h\til{s}_{12} + o_{L^2}(h)\]
with $\til{r}_{12}, \til{r}_{12}', \til{s}_{12}, \til{s}_{12}'\in L^2(M_0)$ independent of $h$ and $\|r_2\|_{L^2} +\|s_2\|_{L^2} = o(h)$.

Plug these solutions into the integral identity we have that 
\[0 =  \int_{M_0} (V_2 - V_1 )|F_A|^4  (e^{2i\psi/h}|F_A|^{-2} |a|^2 + e^{-2i\psi/h} |F_A|^2 |a|^2 + g_0+ hg_1) + o(h)\]
for some $g_0, g_1\in L^2(M_0)$ independent of $h$. 
\begin{lemma}
\label{stationary phase}
In the limit as $h\to0$ the following asymptotic holds:
\begin{eqnarray*}
&&\int_{M_0}(V_2 - V_1 )|F_A|^2  e^{2i\psi/h} |a|^2+ (V_2 - V_1 )|F_A|^6  e^{-2i\psi/h} |a|^2\\&=& hC_+ e^{2i\psi(p_0)/h}(V_2-V_1)(p_0) +hC_- e^{-2i\psi(p_0)/h}(V_2-V_1)(p_0) + o(h).
\end{eqnarray*}
Here $C_+$ and $C_-$ are non-zero constants independent of $h$.
\end{lemma}

Using Lemma \ref{stationary phase} we have that
\[0 = \int_{M_0} (V_1 -V_2) g_0 + O(h)\]
and therefore
\[0 =  \int_{M_0} (V_2 - V_1 )|F_A|^4  (e^{2i\psi/h}|F_A|^{-2} |a|^2 + e^{-2i\psi/h} |F_A|^2 |a|^2 + hg_1) + o(h).\]
Using Lemma \ref{stationary phase} again we get that
\[0 = C_+ e^{2i\psi(p_0)/h} (V_2-V_1)(p_0) +  C_- e^{-2i\psi(p_0)/h}(V_1-V_2)(p_0) + \int_{M_0}(V_1-V_2)|F_A|^4g_1 + o(1)\]
for constants $C_\pm$ independent of $h$. Since $\psi(p_0)\neq 0$ we can choose a sequence of $h\to 0$ such that $ e^{2i\psi(p_0)/h}=  e^{-2i\psi(p_0)/h} =1$ and another sequence $h \to 0$ such that $ e^{2i\psi(p_0)/h}=  e^{-2i\psi(p_0)/h} =-1$ to obtain \[\int_{M_0}(V_1 -V_2) |F_A|^4g_1 = 0.\] 
Therefore, we have that 
\[0 = C_+ e^{2i\psi(p_0)/h} (V_2-V_1)(p_0) +  C_- e^{-2i\psi(p_0)/h}(V_1-V_2)(p_0) + o(1).\]
Again we choose a sequence $h\to 0$ such that $ e^{2i\psi(p_0)/h}= i$ and another sequence $h\to 0$ such that $ e^{2i\psi(p_0)/h}=  e^{-2i\psi(p_0)/h} =1$ we can obtain $(V_1-V_2)(p_0) = 0$.

In order to complete the proof we must provide the\\
\noindent{\bf Proof of Lemma \ref{stationary phase}.} Let $\chi$ be a smooth cutoff function on $M_0$ which is identically $1$ everywhere 
except inside a small ball containing $p_0$ and no other critical point of $\Phi$, and $\chi=0$ near $p_0$.
Setting $V:= V_2- V_1$ we split the oscillatory integral in two parts:
\[\begin{gathered}
\int_{M_0} (e^{2i\psi/h}|F_A|^2 + e^{-2i\psi/h}|F_A|^6)V |a|^2 = \int_{M_0}\chi (e^{2i\psi/h}|F_A|^2 + e^{-2i\psi/h}|F_A|^6)V |a|^2\\+ \int_{M_0}(1-\chi) (e^{2i\psi/h}|F_A|^2 + e^{-2i\psi/h}|F_A|^6)V |a|^2 
 \end{gathered}\]
The phase $\psi$ has nondegenerate critical points, therefore, a standard application of the stationary phase at $p_0$ gives
\[ \int_{M_0}(1-\chi) (e^{2i\psi/h}|F_A|^2 + e^{-2i\psi/h}|F_A|^6)V |a|^2 =hC_+ e^{2i\psi(p_0)/h}V(p_0) +hC_- e^{-2i\psi(p_0)/h}V(p_0) + o(h).\]
Define the potential $\til{V}(\cdot):=V(\cdot)-V(p_0)\in C^{1,\alpha}(M_0)$, then we show that 
\begin{equation}\label{o of h}
\int_{M_0}(1-\chi) (e^{2i\psi/h}|F_A|^2 + e^{-2i\psi/h}|F_A|^6)\til{V}|a|^2=o(h).
\end{equation}
Indeed, first by integration by parts and using $\Delta_g\psi=0$ one has
\[\begin{split}
\int_{M_0}(1-\chi) e^{2i\psi/h} |F_A|^2\til{V} |a|^2=& 
\frac{h}{2i}\int_{M_0}\cjg de^{2i\psi/h} ,d\psi\cjd |F_A|^2\til{V} \frac{(1-\chi)|a|^2}{|d \psi|^2} {\rm{dv}}_g\\
=& \frac{h}{2i}\int_{M_0}e^{2i\psi/h} \cjg d\Big(\frac{(1-\chi)|F_A|^2 |a|^2\til{V}}{|d\psi|^2}\Big),d\psi\cjd
\end{split}\]
and
\[\begin{split}
\int_{M_0}(1-\chi) e^{-2i\psi/h} |F_A|^6\til{V} |a|^2=& 
\frac{-h}{2i}\int_{M_0}\cjg de^{-2i\psi/h} ,d\psi\cjd |F_A|^6\til{V} \frac{(1-\chi)|a|^2}{|d \psi|^2} {\rm{dv}}_g\\
=& \frac{-h}{2i}\int_{M_0}e^{-2i\psi/h} \cjg d\Big(\frac{(1-\chi)|F_A|^6 |a|^2\til{V}}{|d\psi|^2}\Big),d\psi\cjd
\end{split}\]
but we can see that $\cjg d((1-\chi)|F_A|^k |a|^2\til{V}/|d\psi|^2),d\psi\cjd\in L^1(M_0)$: 
this follows directly from the fact that $\til{V}$ is in the H\"older space $C^{1,\alpha}(M_0)$ and $\til{V}(p_0)=0$,
and from the non-degeneracy of ${\rm Hess}(\psi)$.
It then suffice to observe that $\int e^{\pm2i\psi/h} f = o(1)$ for all $f\in L^1(M_0)$ to conclude that \eqref{o of h} holds.
Using similar argument, we now show that
\[\int_{M_0}\chi (e^{2i\psi/h} |F_A|^2+ e^{-2i\psi/h}|F_A|^6)V |a|^2   = o(h).\]
Indeed, since $a$ vanishes to large order at all boundary critical points of $\psi$, we may write
{\tiny\[\begin{split}
\int_{M_0}\chi (e^{2i\psi/h}|F_A|^2 + e^{-2i\psi/h}|F_A|^6)V |a|^2 {\rm{dv}}_g =& \frac{h}{2i}\int_{M_0}\big(\cjg d e^{2i\psi/h} ,d\psi\cjd  |F_A|^2 -\cjg d e^{-2i\psi/h},d\psi\cjd  |F_A|^6\big)\frac{\chi V|a|^2}{|d \psi|^2}\\
=&-\frac{h}{2i}\int_{M_0}(e^{2i\psi/h} {\rm div}_g\Big(V\frac{\chi|F_A|^2|a|^2}{|d\psi|^2}\nabla^g\psi\Big)- e^{-2i\psi/h}{\rm div}_g\Big(V\frac{\chi|F_A|^6|a|^2}{|d\psi|^2}\nabla^g\psi\Big)) \\
&+ \frac{h}{2i}\int_{\Gamma} (e^{2i\psi/h} - e^{-2i\psi/h})V \frac{|a|^2}{|d\psi|^2}\pl_\nu\psi\,.
\end{split}\]}
Here the expression for the boundary integral is obtained by using the fact that $V_1 = V_2$ on $\partial M_0 \backslash \Gamma_0$ from boundary determination and $|F_A| = 1$ on $\Gamma_0$ by construction. 

For the interior integral we use the fact that $\int e^{\pm2i\psi/h} f = o(1)$ for all $f\in L^1(M_0)$ to conclude that 
\[-\frac{h}{2i}\int_{M_0}(e^{2i\psi/h} {\rm div}_g\Big(V\frac{\chi|F_A|^2|a|^2}{|d\psi|^2}\nabla^g\psi\Big)- e^{-2i\psi/h}{\rm div}_g\Big(V\frac{\chi|F_A|^6|a|^2}{|d\psi|^2}\nabla^g\psi\Big)) = o(h)\]
and for the boundary integral, 
we observe that on $\Gamma$, $\psi = 0$ by construction so 
$(e^{2i\psi/h} - e^{-2i\psi/h}) = 0$.  Therefore
\[\int_{M_0}\chi (e^{2i\psi/h} + e^{-2i\psi/h})V |a|^2 {\rm{dv}}_g = o(h)\] 
and the proof is complete.\qed

\end{subsection}
\end{section}

\end{document}